\documentclass{amsart}
\usepackage{amsmath}
\usepackage{amssymb}
\usepackage{amsfonts}
\usepackage{mathrsfs}
\usepackage{graphics}
\usepackage{graphicx}

\title[Mixing for Time-Changes of Heisenberg Nilflows]%
{Mixing for Time-Changes \\ of Heisenberg Nilflows}

\author{Artur Avila} \author{Giovanni Forni} \author{Corinna Ulcigrai}

\address{CNRS UMR 7586, Institut de Math\'ematiques de Jussieu,
175 rue du Chevaleret, 75013-Paris, FRANCE    }

\address{Department  of Mathematics\\
  University of Maryland \\
  College Park, MD USA}
  
\address{School of Mathematics \\ University of Bristol \\ Bristol, UK}
  
\begin{document}

\def\Leb{{\mathrm{Leb}}}
\def\PP{{\mathcal{P}}}
\def\RR{{{\mathcal{R}}}}
\def\MM{{\mathcal{M}_f}}
\def\TT{{\mathcal{T}_f}}
\def\R{{\mathbb{R}}}
\def\Q{{\mathbb{Q}}}
\def\C{{\mathbb{C}}}
\def\Z{{\mathbb{Z}}}
\def\N{{\mathbb{N}}}
\def\T{{\mathbb{T}}}
\def\Leb{\mathrm{Leb}}
\def\Re{{\operatorname{Re}}}
\def\Heis {{N}}
\def\heis{{\mathfrak n}}

\def\be{\begin{equation}}
\def\ee{\end{equation}}
\def\bes{\begin{equation*}}
\def\ees{\end{equation*}}
\newcommand{\Sym}[1]{\mathcal{S} _{ #1} }
\newcommand{\ud}{\, \mathrm{d}}

\newenvironment{proofof}[2]{\begin{proof}[Proof of #1 \ref{#2}.]}{\end{proof}}
\newtheorem{thm}{Theorem}
\newtheorem{lemma}{Lemma}
\newtheorem{cor}{Corollary}
\newtheorem{defn}{Definition}
\newtheorem{prop}{Proposition}
{\theoremstyle{remark} \newtheorem{rem}[thm]{Remark}}

\begin{abstract}
We consider reparametrizations of Heisenberg nilflows. We show that if a Heisenberg nilflow is uniquely ergodic, all non-trivial time-changes within a dense subspace of smooth time-changes are mixing. Equivalently, in the language of special flows, we consider special flows over linear skew-shift
extensions of irrational rotations of the circle. Without assuming any Diophantine condition on the frequency, we define a dense class of smooth roof functions for which the corresponding special flows are mixing. Mixing is produced by a mechanism known as stretching of Birkhoff sums.  The complement of the set of mixing time-changes (or, equivalently, of mixing roof functions) has countable codimension and can be explicitely described in terms of the invariant distributions for the nilflow (or, equivalently, for the skew-shift), allowing to produce concrete examples of mixing time-changes. 
\end{abstract}

\maketitle

\section{Introduction}

In this paper we give a contribution to the smooth ergodic theory of parabolic flows. We prove
that for any uniquely ergodic Heisenberg nilflow all non-trivial time-changes, within 
a dense subspace of time-changes, are mixing. The set of trivial time-changes has countable codimension and can be explicitly described in terms of invariant distributions for the nilflow.

A non-singular flow is called \emph{parabolic} if nearby orbits diverge polynomially in time. If nearby orbits diverge exponentially, the flow is called \emph{hyperbolic}; if there is no divergence (or perhaps 
it is slower than polynomial) the flow is called \emph{elliptic}.  In contrast with the hyperbolic
case, and to a lesser extent with the elliptic case, there is no general theory which describes 
the dynamics of parabolic flows. The main (typical) ergodic properties often associated with parabolic dynamics are unique ergodicity, mixing, polynomial speed of convergence of ergodic averages and polynomial decay of correlations for smooth functions and, of course, zero entropy. Another
important feature of parabolic flows is the presence of  infinitely many independent distributional obstructions to the solution of the so-called \emph{cohomological equation}  (which are not
signed measures as in the hyperbolic case). This important property allows for the existence 
of non-trivial time-changes which are not given by the existence of fast periodic approximations
(Liouvillean phenomenon) as in the classical, better understood, elliptic case.

A fundamental example of a parabolic flow is given by \emph{horocycle flows} on compact negatively 
curved surfaces. It is well known that horocycle flows are uniquely ergodic \cite{Fu:uni}, mixing of all orders \cite{Ma:hor} and have countable Lebesgue spectrum \cite{Pa:hor}.  
Kuschnirenko \cite{Kuschnirenko:mixing} has proved that  all time-changes are mixing
under an explicit condition which holds if the time-change is sufficiently small (in the $C^1$ topology). 
It is not known whether this results extends to all smooth time-changes. Nothing is known about
the spectral properties of time-changes.  A. Katok has conjectured that countable Lebesgue
spectrum persists at least under Kuschnirenko's condition.

Other important examples of flows which are sometimes considered parabolic is given by \emph{area-preserving flows on surfaces} of higher genus (genus greater than two) with saddle-like singularities.  
In this case the orbit divergence is entirely produced by the splitting of trajectories near the singularities.
In particular, directional flows on translation surfaces, often called \emph{translation flows}, which appear in the study of the geodesic flow on a surface endowed with a flat metric with conical singularities (we refer for example to the survey  \cite{Ma:rat} for definitions) have been studied
in depth in the past thirty years.  The unique ergodicity of almost any translation flow is a fundamental result of H. Masur \cite{Ma:int} and W. Veech \cite{Ve:gau}, while the first two authors proved that typical translation flows are weak mixing \cite{AF:wea}. Translation flows are never mixing, as known since the work of Katok \cite{Ka:int}.  This leads to the question of mixing in reparametrizations of translation flows. 

Time-changes of translation flows can be represented as special flows over interval exchange transformations (IET's), which are one-dimensional piecewise isometries.  A reparametrization of translation  flows which appear naturally in physical problems is the locally Hamiltonian parametrization, which was  studied since Novikov and his school in the Nineties. The corresponding flows on surfaces are known as flows given by a \emph{multi-valued Hamiltonian} and can be represented as special flows over IET's with a roof function which has singularities. If the zeros 
of the flow are degenerate, i.e.~ they are multi-saddles, they give rise to power-like  singularities of the roof functions, if they are non-degenerate (Morse) saddles,  they give rise to logarithmic singularities. If the flow has saddle loops, logarithmic singularities are typically asymmetric, otherwise they are symmetric.

The mixing properties of special flows have been studied in depth by many authors.  The situation can perhaps be summarized as follows. On one hand, weak mixing is typical and it does not require any assumptions on the singularity of the roof functions: the result of the first two authors  \cite{AF:wea} already mentioned above is that for any piece-wise constant roof function, weak mixing holds for typical IET's. The third author proved that a simple mechanism allows to show weak mixing in the case of roof functions with  logarithmic singularities over typical IET's. 

On the other hand, mixing relies crucially on the presence of singularities. Indeed, for roof functions of bounded variations (thus in particular for smooth roofs) A. Katok  \cite{Ka:int} proved absence of mixing.    Ko{\v c}ergin proved in \cite{Ko:mix} that a flow given by a roof function with \emph{power-like singularities} over a typical IET (with minimal combinatorics) is mixing and mixing is produced as an effect of the shear at the singularities. When the singularities are \emph{logarithmic},  the symmetry conditions in fact leads to the mutual cancellation of the mixing effect of the saddles. Thus,  mixing depends on whether the singularities are symmetric or not. In the \emph{asymmetric} case, typical mixing was proved by Khanin and Sinai \cite{SK:mix} for flows over circle rotations and by the third author for flows over IET's on any number of intervals \cite{Ul:mix}. In the \emph{symmetric }case, Ko{\v c}ergin proved the absence of mixing for flows over circle rotations \cite{Ko:abs}. This result was extended  to typical IET's first by Scheglov  \cite{Sch:abs}, who treated the case of IET's of
four and five intervals, and finally to typical IET's on any number of intervals by the third author \cite{Ul:abs}.

Another important class of (homogenous) parabolic flows is given by \emph{nilflows}. By classical results of homogenous dynamics, see \cite{AGH}, minimal nilflows are uniquely ergodic. However, in constrast with horocycle flows, they are never mixing, not even weak mixing. However, there is a clear geometric obstruction to the (weak) mixing property, that is, every nilflow is only partially
parabolic, in the sense that it has an elliptic factor given by a linear flow on a torus. For observables in the 
orthogonal complement of the span of the pull-back to the nilmanifold of the toral characters, any 
nilflows has countable Lebesgue spectrum \cite{Green, AGH}, hence it is mixing. Thus, nilflows
have the properties of \emph{relative} Lebesgue spectrum and mixing.

Our result  confirm some heuristic principles on the dynamics of parabolic flows. In particular,  
for time-changes of any Heisenberg nilflow (without Diophantine conditions) mixing is prevalent 
and it occurs unless the flow is only partially parabolic (presence of a measurable
elliptic factor), which in this case means that the time-change is trivial. As a consequence,
weak and strong mixing are equivalent. This picture is an agreement with a conjectural
generalization of Kuschnirenko  mixing result \cite{Kuschnirenko:mixing} to all time-changes
of the horocycle flow. It shows that Heisenberg nilflows differ significantly from translation
flows or area-preserving flows on higher genus surfaces. As outlined above, in the latter
case, the typical (non-trivial) time-change is weak mixing, but not mixing, and mixing
can only be produced by shear at the singularities. In other words, area-preserving flows
on surfaces are better classified as elliptic flows with singularities than as parabolic flows.

Our approach to mixing for nilflows has the advantage of not requiring Diophantine conditions. 
However, it does not seem to be possible to derive quantitative informations on the decay of correlations. A natural conjecture is that if the elliptic toral factor is a Diophantine linear flow, then 
the decay of correlations of smooth functions is \emph{polynomial} in time. This conjecture is 
consistent with Ratner's result \cite{Ra:rat} on the decay of correlations for horocycle flows and with 
the rate of relative mixing for Heisenberg nilflows (which can be estimated by Fourier analysis). 
In fact, several results on parabolic flows suggest the following heuristic principle: \emph{a uniquely ergodic smooth flow with polynomial speed of convergence of ergodic averages is a smooth 
time-change of a smooth flow with polynomial decay of correlations (for smooth functions)}. 
For horocycle flows, the rate of mixing \cite{Ra:rat} as well as  the speed of convergence of ergodic 
averages \cite{Za, Sa, Bu, Hj, FlaFor:hor, St} are polynomial. 
For minimal ergodic area-preserving flows and translation flows, the polynomial decay of
ergodic averages (for smooth functions vanishing at sufficiently high order at the singularities) 
was conjectured by A. Zorich \cite{Zorich} and M. Kontsevich \cite{Kontsevich} and proved by the 
second author in \cite{Fo:dev}. According to the above-mentioned heuristic principle, the decay of correlations for time-changes with a degenerate saddle, should also be polynomial. While mixing 
is known after Ko{\v c}ergin's result \cite{Ko:mix}, to the authors' best knowledge polynomial decay
of correlations (under a Diophantine condition) has been proved only for the particular case of flows 
on the $2$-torus with a single degenerate saddle of restricted type  \cite{Fayad:pol}.
For Heisenberg nilflows, the speed of convergence of ergodic averages of smooth functions is polynomial and the optimal exponents (which depend on the Diophantine properties of the toral factor) are known \cite{FlaFor:nil}. This result is related to optimal  bounds for Weyl sums of quadratic polynomials, see \cite{FJK:Weyl, Mr}.  According to the heuristics proposed above, there should
be mixing time-changes with polynomial decay of correlations.

On the \emph{spectral properties} of our mixing time-changes of Heisenberg nilflows, it is reasonable to conjecture that  they have countable Lebesgue spectrum. However, this seems a difficult problem, which most likely cannot be approached through estimates on the correlation decay. As A.~Katok 
has observed, this difficulty already appears for the horocycle flow and its time-changes. 

The mechanism that we use to produce mixing is sometimes known as \emph{stretching of Birkhoff sums}. The stretching of Birkhoff sums for Heisenberg nilflows is derived from a theorem on the growth 
of Birkhoff sums of functions which are not coboundaries with a \emph{measurable }transfer function. 
This result is quite general and can be proved for all nilflows. In fact, it is essentially based on a measurable Gottschalk-Hedlund theorem, which holds for any volume preserving uniquely ergodic dynamical system, and on the parabolic divergence of orbits (although in a quite explicit form).
Finally, we prove a theorem on cocycle effectiveness for the Heisenberg case, which states
that if a smooth function is a coboundary with a measurable transfer function, then the transfer 
function is in fact smooth. This result is based on sharp bounds for ergodic sums which are
only available in the Heisenberg case \cite{FJK:Weyl, Mr, FlaFor:nil}. The cocycle effectiveness
allows a concrete description of mixing time-changes in terms of the non-vanishing of any of
the distributional obstructions to the existence of smooth solutions of the cohomological equation.

It is worth recalling that a similar mixing mechanism was used by Fayad in \cite{Fa:ana} to produce smooth (analytic) mixing time-changes of some elliptic flows, i.e.~linear flows on tori $\T^n$, with $n
\geq 3$ and Liouvillean frequencies. If $n=2$, smooth time-changes of a linear flow on $\T^2$ are never mixing (for example, as a consequence of the result of A. Katok \cite{Ka:int} quoted above). Moreover, for Diophantine linear flows on $\T^n$  all smooth time-changes are trivial (since all smooth function of zero average are smooth coboundaries) by the generalization to all dimensions \cite{He} of a well-known theorem of Kolmogorov \cite{Kolmogorov}. In dimension 
$n=2$, the Denjoy-Koksma inequality explains the  absence of mixing time-changes even for
Liouvillean frequencies, but does not prevent the existence of weak mixing examples, which
are in fact topologically generic, as proved in \cite{Fayad:wm}. In higher dimensions, the failure
of the Denjoy-Koksma inequality opens the way for mixing examples with Liouvillean frequency
\cite{Fa:ana}. Thus, in this elliptic realm, the phenomenon of stretching of Birkhoff sums and mixing 
time-changes is not generic and can occurr only for Liouvillean frequencies, in  contrast with our 
result for nilflows, where mixing time-changes  are generic for any uniquely ergodic nilflow, or equivalently, as long as the frequency of the elliptic factor is irrational.

\subsection*{Outline}
In Section~\ref{sec:definitions and results} we give the definitions of Heisenberg nilflows (\S~\ref{sec:heisenberg}), special flows (\S~\ref{sec:mixspecialflows}) and time-changes (\S~\ref{sec:mixingtimechanges}) and recall how to represent a Heisenberg nilflow as a special flow~(\ref{sec:returnmaps}). We then state our main results for  time-changes of nilfows in~\S~\ref{sec:mixingtimechanges} (Theorem~\ref{thm:mainflows})  and in~\S~\ref{sec:mixspecialflows} in the language of special flows (Theorem~\ref{thm:main}). The class of mixing time-changes is defined  
in~\S~\ref{sec:mixingroofs} (Definition~\ref{mixing_class}) and, as explained in~\S~\ref{sec:effectiveness}, it can be explicitly characterized in terms of invariant distributions for the nilflow (Theorem~\ref{thm:effectiveness}).  Sections~\ref{growth_proof_sec},~\ref{mixing_proof_sec},~\ref{sec:effectivenessproof} and~\ref{proof:mainthms} are devoted to proofs: in Section~\ref{growth_proof_sec} we prove that non triviality of the time-change guarantees that there is stretch of Birkhoff sums (Theorem~\ref{growth}). Using this stretch, in Section~\ref{mixing_proof_sec} we  implement the  mixing mechanism and prove mixing (Theorem~\ref{mixing}). Section~\ref{sec:effectivenessproof} contains the proof  of the effective characterization of non-trivial time-changes (Theorem~\ref{thm:effectiveness}) which allows to exhibit explicit examples of mixing time-changes. The proofs of Theorem~\ref{thm:mainflows} and Theorem~\ref{thm:main} then follow easily  in Section~\ref{proof:mainthms}. 

\newpage
\section{Definitions and main results.}\label{sec:definitions and results}

\subsection{Heisenberg nilflows}
\label{sec:heisenberg}
The $3$-dimensional \emph{Heisenberg  group }$\Heis$ is the unique connected, simply connected
Lie group with $3$-dimensional Lie algebra $\heis$ on two generators $X$, $Y$ satisfying 
the Heisenberg commutation relations
$$
[X,Y] =Z \,, \quad [X,Z]= [Y,Z]=0\,.
$$
Up to isomorphisms, $\Heis$ is the group of upper triangular  unipotent matrices
\begin{equation}\label{eq:matrix}
[x,y,z] :=  \begin{pmatrix}
  1 & x & z\\
  0 &1 & y\\
  0& 0 & 1
\end{pmatrix}, \qquad x,y,z \in \R.
\end{equation}
A basis of the Lie algebra $\heis$ satisfying the Heisenberg commutations relations is given by
the matrices
\begin{equation}\label{eq:nilbasis}
X=\begin{pmatrix}
  0 & 1 & 0\\
  0 &0 & 0\\
  0& 0 & 0
\end{pmatrix}, \quad
Y=\begin{pmatrix}
  0 & 0 & 0\\
  0 &0 & 1\\
  0& 0 & 0
\end{pmatrix},\quad 
Z=\begin{pmatrix}
  0 & 0 & 1\\
  0 &0 & 0\\
  0& 0 & 0
\end{pmatrix}.
\end{equation}
The abelianized Lie algebra $\heis/[\heis, \heis]$ of the Heisenberg Lie algebra is isomorphic to $\R^2$ (as a Lie algebra), hence the abelianized Lie group $\Heis/[\Heis, \Heis]$ of the Heisenberg group is isomorphic to $\R^2$ (as a Lie group). In fact, both the center $Z(\Heis)$ and the commutator subgroup $[\Heis, \Heis]$ of the Heisenberg group $\Heis$ are equal to the one-parameter subgroup 
$\{[0,0,z] : r\in \R\}$ and the maps
\begin{equation}\label{eq:nonsplitmaps}
z\mapsto [0,0,z]
\qquad \text{ and }\qquad
[x,y,z]\mapsto (x,y)\,.
\end{equation}
define a (non-split) exact sequence
\begin{equation}\label{eq:nonsplitseq}
0 \to \R \to \Heis \to \R^2\to 0\,,
\end{equation}
which exhibits $\Heis$ as a line bundle over $\R^2$.

\smallskip
A compact \emph{Heisenberg nilmanifold} is the quotient $M:= \Gamma \backslash \Heis$ 
of the Heisenberg group over a co-compact lattice $\Gamma <\Heis$. It is well-known that there 
exists a positive integer~$E\in \N$ such that, up to an automorphism of~$\Heis$, the lattice
$\Gamma$ coincide with the lattice
\begin{equation*}
\Gamma:= \left\{
\left. \begin{pmatrix}
  1 & x & z/E \\
  0 &1 & y\\
  0& 0 & 1
\end{pmatrix}\right.  :   \,\, x,y,z \in \Z\right\}.
\end{equation*}
Let $\overline\Gamma:= \Gamma/ [\Gamma, \Gamma]<\R^2$ denote the abelianized lattice. 
The canonical projection homomorphism $\Heis  \to  \Heis/ [\Heis, \Heis] \approx \R^2$ defined in 
(\ref{eq:nonsplitmaps}) induces a Seifert fibration $\pi: M \to \T^2 = \overline\Gamma\backslash 
\R^2$, that is, $M$ is a circle bundle over the $2$-torus $\T^2=\R^2/\Z^2$ with fibers given by the 
orbits of the flow by right translation of the central one-parameter subgroup $Z(\Heis)=\{ \exp (zZ)\}_{z\in \R}$. The left invariant fields $X$, $Y$ on $M$ define a connection whose total curvature (the Euler characteristic of the fibration) is exactly~$E$. Any Heisenberg nilmanifold $M$ has a natural probability measure $\mu$ locally given by the Haar measure of~$\Heis$.

The group $\Heis$ acts on the right transitively on $M$ by right multiplication:
\begin{equation*}
  R_g ( x) := x\,g, \quad x \in M, \, g \in \Heis.
\end{equation*}
By definition, \emph{Heisenberg nilflows }are the flows obtained by the restriction of this right action to the one-parameter subgroups on $\Heis$.  The measure $\mu$ defined above, which is invariant for 
the right action of $\Heis$ on $M$, is, in particular, invariant for all nilflows on $M$.

Thus each $W:=w_x X+ w_y Y + w_z Z \in \heis$ defines a measure preserving flow 
$( \phi_W , \mu) $ on $M$ where $\phi_W:= \{ \phi_W^t\}_{t  \in \R} $ is given by the formula
\begin{equation*}
  \phi_W^t (x) = x \exp (tW) , \quad x \in M\ , t \in \R \,.
\end{equation*}
The projection $\bar W$ of $W$ into $\R^2$ is the generator of a linear flow $\psi_{\bar W} := \{ \psi_{\bar
  W}^t\}_{t\in\R}$ on $\T^2 \approx  \R^2\backslash \overline \Gamma$ defined by
\begin{equation*}
  \psi_{\bar W}^t (x,y ) =(x+t w_x, y+t w_y).
\end{equation*}
The canonical projection $\pi:M \to \T^2$ intertwines the flows $\phi_W$ and $ \psi_{\bar W} $. 
We recall the following basic result:

\begin{thm} \cite{Green, AGH} \label{prop:Green}
  The following conditions are equivalent: 
  \begin{enumerate}
  \item The nilflow $( \phi_W, \mu) $ is ergodic.
  \item The nilflow $\phi_W $ is uniquely ergodic.
  \item The nilflow $\phi_W$ is minimal.
  \item The projected flow $\psi_{\bar W} $ is an
    irrational linear flow on $\T^2$ and hence it is minimal and uniquely ergodic.
\end{enumerate}
\end{thm}
Results on the speed of equidistribution of Heisenberg nilflows for smooth functions
were proved in \cite{FlaFor:nil} by L.~Flaminio and the second author. Similar results can be proved by bounds on Weyl sums for quadratic polynomials, see \cite{FJK:Weyl, Mr}.

\smallskip
Nilflows are clearly not weak mixing, hence not mixing. In fact, all eigenfunctions of linear
toral flows (that is, all characters of the group $\T^2$) pull-back to eigenfunctions
of all nilflows on $M= \Gamma \backslash \Heis$. However, all nilflows are \emph{relatively mixing} 
in the following sense. Let $H:=\pi^\ast L^2(\T^2) \subset L^2(M)$ be the subspace
obtained by pull-back of the square-integrable functions on the torus $\T^2$
and let $H^\perp \subset L^2(\T)$ its orthogonal complement. The following result holds. 

\begin{thm} \cite{Green, AGH} 
The restriction of any nilflow  $( \phi_W , \mu) $ 
to the $\Heis$-invariant subspace $H^\perp \subset L^2(\T)$ has countable Lebesgue 
spectrum, hence it is mixing.
\end{thm}

In fact, it is possible to prove by the theory of unitary representations of the Heisenberg
group (the Stone-Von Neumann theorem, see for example \cite{CG:rep}, \S 2.2)  that for all 
sufficiently smooth functions in $H^\perp$ the decay of correlations is polynomial (it
is faster than any polynomial for infinitely differentiable functions in $H^\perp$).

\subsection{Return maps of Heisenberg nilflows} \label{sec:returnmaps} Any uniquely ergodic Heisenberg nilflow has a smooth compact transversal surface, isomorphic to a $2$-dimensional torus. One can compute the return map and the return time function (see \cite{Sta:dyn}, \S3). It turns out that the return time is constant and the return map is a linear \emph{skew-shift} over an irrational rotation of the circle.  We recall this well-known construction for the convenience of the reader.

\smallskip
Let $\Sigma \subset M$ be the smooth surface defined as follows:
\bes
\Sigma := \{ \Gamma \exp (x X + z Z ) : \,\,  (x,z) \in \R^2 \} \,.  
\ees
Since the subspace $<X,Z>$ generated in $\heis$ by $X,Z\in \heis$ is an abelian
ideal, the surface $\Sigma$ is isomorphic to a $2$-dimensional torus. The isomorphism is given by the map
\bes
j(x,z) =  \Gamma \exp (xX + z Z) \,, \quad \text{for all } \, (x,z) \in 
\T^2_E := \R^2 /  (\Z \times \Z/E)\,.
\ees

Let $W:= w_xX + w_yY + w_z Z$ be the  generator of a uniquely ergodic nilflow
and let $\phi^W= \{\phi^W_t\}_{t\in \R}$ denote the corresponding Heisenberg nilflow.

\begin{lemma} 
\label{lemma:returnmap}
The first return time function of the flow $\phi^W$ to the 
transverse section $\Sigma$ is constant equal to $1/w_y$ and the
first return (Poincar\'e) map $ P_W:\Sigma \to \Sigma$ is given by the following formula:
\be
\label{eq:returnmap}
P_W \circ j (x,z) =  j(x + \frac{w_x}{w_y}, z + x +  \frac{w_z}{w_y} +  \frac{w_x}{2w_y}) 
\,, \quad \text{ for all } (x,z) \in \T^2_E\,.
\ee
\end{lemma}
\begin{proof}  Since the nilflow is uniquely ergodic, we have $w_y\neq 0$, which implies that 
the surface $\Sigma$ is transverse to the nilflow. The set of all return times of the nilflow to 
$\Sigma$ is a subset of the set of all return times of the projected linear flow $\psi_{\bar W}$ on the torus 
$\T_\Gamma$, which is equal to the subgroup $ \Z / w_y\subset \R$. Finally, by the Baker-Campbell-Hausdorff formula, since $[\heis, [\heis,\heis]]=0$, we have
\bes
\begin{aligned}
 \exp (- Y)& \exp(xX + z Z) \exp \left( W/w_y\right)  =   \\ & =  \exp\left[ (x + \frac{w_x}{w_y}) X + 
 (z + x +   \frac{w_z}{w_y} +  \frac{w_x}{2w_y}  ) Z \right] \,.
 \end{aligned}
\ees
Since by definition $ \exp (- Y) \in \Gamma$, it follows from the above formula that the 
forward first return time is equal to $1/w_y$ for all $(x,z) \in \T^2_E$ and that the 
forward first return time map is given by formula~(\ref{eq:returnmap}) as claimed.
\end{proof}

Lemma~\ref{lemma:returnmap} implies that any (uniquely ergodic) Heisenberg nilflow 
is smoothly isomorphic to a \emph{special flow }over a linear skew-shift  of the form 
(\ref{eq:returnmap}) with constant \emph{roof function}. The notion of a special flow
is recalled below in Section~\ref{sec:mixspecialflows}.  

\subsection{Mixing time-changes}\label{sec:mixingtimechanges}

We recall below basic notions about time-changes of flows and state our main theorem
on mixing of time-changes of Heisenberg nilflows.

A flow $\{\widetilde{h}_t \}_{t \in \R}$ is called a \emph{reparametrization} or a  
\emph{time-change }of a flow  $\{{h}_t \}_{t \in \R}$ on $X$  if there exists a measurable
function $\tau: X \times \R \rightarrow \R$  such that for all  
$x \in X$ and $t \in \R$ we have $\widetilde{h}_t(x) = h_{\tau(x,t)}(x) $. Since
$\{\widetilde{h}_t \}_{t \in \R}$ is assumed to be a flow (a one-parameter group)
the function $\tau(x,\cdot) : \R \rightarrow \R$ is an \emph{additive cocycle}
over the flow $\{\widetilde{h}_t \}_{t \in \R}$, that is, it satisfies the cocycle identity:
$$
\tau(x, s+t ) = \tau( \widetilde{h}_s(x), t) + \tau(x,s)\,, \quad \text{ for all }  x\in X\,, \,\, s,t \in \R\,.
$$
If $X$ is a manifold and $\{{h}_t \}_{t \in \R}$  is a smooth flow, we will say that 
$\{\widetilde{h}_t \}_{t \in \R}$ is a smooth reparametrization if the cocycle $\tau$ is a 
smooth function. By the cocycle property a smooth cocycle is uniquely determined by
its infinitesimal generator, that is the function $\alpha_\tau: X \to \R$ defined by the formula:
$$
\alpha_\tau(x) := \frac{\partial \tau}{\partial t}(x,0) \,, \quad \text{ for all } x\in X\,. 
$$
In fact, given any positive function $\alpha: X \to \R^+$, the formula
$$
\tau_\alpha (x,t) := \int_0^t  \alpha(\tilde h_s(x)) ds \,, \quad \text{ for all}\, (x,t)\in X\times \R\,
$$
is cocycle over the flow $\{\tilde h_t\}_{t\in \R}$ with infinitesimal generator $\alpha$.

The infinitesimal generators $\widetilde V$ and $V$ of the flows $\{\widetilde{h}_t \}_{t \in \R}$ and $\{{h}_t \}_{t \in \R}$  respectively are related by the identity:
$$
\widetilde V:=\left. \frac{d\widetilde{h}_t}{dt} \right\vert_{t=0} =\left. \alpha_\tau\, \frac{d{h}_t}{dt} \right\vert_{t=0} := 
\alpha_\tau V  \,.
$$

An additive cocycle $\tau:X\times \R \to \R$ over the flow $\{\widetilde{h}_t\}_{t\in \R}$ is called a measurable (respectively smooth)  \emph{coboundary} if there exists a measurable (respectively smooth)  function $u:X \to \R$, called the \emph{transfer function},  such that
$$
\tau(x,t) =  u \circ \widetilde{h}_t(x) - u(x)\,, \quad \text { for all } (x,t)\in X\times \R\,.
$$ 
The additive cocycle $\tau$ is a measurable (smooth)  coboundary if and only if its infinitesimal generator $\alpha_\tau$ is a measurable (smooth)  \emph{coboundary} for the infinitesimal generator 
$V$ of the flow $\{\widetilde{h}_t\}_{t\in \R}$, that is, if there exists a measurable (smooth)  function 
$u:X \to \R$, also called the transfer function, such that 
$\widetilde V u = \alpha_\tau$. 
 Two additive cocycles  are said to be measurably (respectively smoothly)  \emph{cohomologous} if their difference is a measurable (respectively smooth)  coboundary in the above sense. 
 A cocycle is said to be an \emph{almost coboundary} if it is cohomologous to a constant cocycle (see \cite{Katok:CC}, Def. 9.4).

An elementary, but fundamental, result establishes that time-changes given by measurably (smoothly)  
cohomologous coycles are measurably (smoothly)  isomorphic (see for example \cite{Katok:CC}, \S9). The regularity of the isomorphisms depends on the regularity of the transfer function. A time-change defined by a measurable (smooth)  almost coboundary is called \emph{measurably (smoothly)  trivial}.

Let $\{h_t\}_{t\in \R}$ be a uniquely ergodic homogeneous flow on the Heisenberg nilmanifold
$M$. For any  function $\alpha:C^\infty(M) \to \R^+$ let $h^{\alpha}:= \{h^\alpha_t\}_{t\in \R}$ be the time-change 
with generator given by the formula
$$
\left.\frac{d{h}^\alpha_t}{dt} \right\vert_{t=0} = \left. \alpha \, \frac{dh_t}{dt} \right\vert_{t=0}    \,.
$$

We recall that a measure preserving flow $\varphi: = \{\varphi_t\}_{t \in \R }$ on a probability space $(X,\mu)$ is said to be 
\emph{weak mixing } if, for each pair of measurable sets $A$, $B \subset X$,
$$
 \lim_{t\to \infty}  \frac{1}{t} \int_0^t \vert \mu( \varphi_s(A)\cap B)- \mu(A)\mu(B) \vert\ud s  = 0\,,
$$
and \emph{mixing} if for each pair of measurable sets $A$, $B$, one has 
$$
\lim_{t\rightarrow \infty} \mu( \varphi_t(A)\cap B)=\mu(A)\mu(B)\,.
$$

\begin{thm}[Mixing time-changes for Heisenberg niflows]
\label{thm:mainflows}
There exists a  subspace $\mathcal T_h  \subset \mathcal A \subset C^\infty(M)$ of countable codimension in a dense subspace  $\mathcal A \subset C^\infty(M)$ such that for any positive 
function $\alpha \in \mathcal A$  the following properties are equivalent:
\begin{enumerate}
\item the function $\alpha \in \mathcal M_h:= \mathcal A\setminus \mathcal T_h  $;
\item the time-change $h^\alpha$ is not smoothly trivial;
\item the time-change $h^\alpha$ is weak mixing;
\item the time-change $h^\alpha$  is mixing.
\end{enumerate}
\end{thm}
Theorem~\ref{thm:mainflows} is proved in~\S~\ref{proof:mainthms}.
Our results leaves open several natural questions on possible generalizations of
Theorem~\ref{thm:mainflows} and on the dynamics of the mixing flows constructed.

\newpage
{\bf Questions.}

a) Does Theorem~\ref{thm:mainflows} holds within the class of all smooth time-changes?

b) Does it extends to  nilflows on $2$-step nilmanifolds on
several generators?

c) Does it  extends to nilflows on $s$-step nilmanifolds 
for any $s\geq 3$? 

d) Is the correlation decay polynomial in time for sufficiently smooth functions (under a Diophantine
conditions on the frequency)?

e) Is the spectrum of mixing time-changes singular continuous or absolutely continuous? 
Is it Lebesgue with countable multiplicity?

\subsection{Mixing special flows over skew shifts  on $\T^2$.}
\label{sec:mixspecialflows}
In this section we recall the notion of a special flow and the representation of time-changes
in terms of special flows. We then state our main theorem for special flows over uniquely
ergodic skew-shifts  on $\T^2$.

Let $f:\Sigma \to \Sigma$ be a Poincar\'e return map of the flow $\{{h}_t \}_{t\in \R}$ on $X$ to 
a measurable transverse section $\Sigma\subset X$ and let $\Phi:\Sigma \to \R^+$ be the \emph{return time 
function} (in general defined only almost everywhere).  The flow $\{{h}_t \}_{t\in \R}$ is isomorphic
to a \emph{special flow} over the map $f:\Sigma \to \Sigma$ with \emph{roof function} $\Phi>0$, defined as  we now recall.
Given any function $\Phi$ on $\Sigma$, let $\Phi_n$ denote the $n^{th}$ Birkhoff sums along the
orbits of the map $f:\Sigma \to \Sigma$, that is, the function 
 \be
 \label{eq:Birkhoffsums}
 \Phi_n:=\sum_{k=0}^{n-1} \Phi \circ f^k\,.
 \ee
 If $\Phi>0$ is a continuous positive function, we let $f^\Phi = \{ f^\Phi_t\}_{t \in \R}$ be the special flow over $f$ with roof function $\Phi$, which is defined as a the quotient of the unit speed vertical flow 
 $\dot{z}=1$ on the phase space $\{(x,z) \in \Sigma \times \R\}$ with respect to the equivalence relation $ \sim_\Phi$ defined by  $(x,\Phi(x)+z) \sim_\Phi (f(x),z)$, for all $x\in \Sigma, z\in \R$. 
The flow $f^{\Phi}$ can thus be seen as  defined on the fundamental domain $\{ (x,z):  \,  x \in \Sigma,\,  
0 \leq z < \Phi(x) \}$, explicitly given by the formula
\be \label{susp flow def}
f^{\Phi}_t \left( x, z \right) = \left( f^{{n_t}(x,z)} (x) , z+ t- \Phi_{{n_t}(x,z)}(x) \right) , 
\ee
where $n_t(x,z)$ is the maximum $n\in \N$ such that $\Phi_n (x)< t+z$.
For any $f$-invariant measure $\nu$ on $\Sigma$, the finite measure obtained by the restriction of the product measure $\nu \times\Leb$  (where $\Leb$ is the Lebesgue measure in the $z$-fiber) to the domain of $f^\Phi$ is invariant by the special flow $f^\Phi$.

 A function $\Phi:\Sigma \to \R$ is called a  measurable (smooth)  \emph{coboundary }for the map 
 $f:\Sigma \to \Sigma$ if and only if there exists a measurable (smooth)  function $u:\Sigma\to \R$, 
 also called the \emph{transfer function}, such that $\Phi= u\circ f -u $. 
 Two functions are called measurably (smoothly)  \emph{cohomologous} if their difference is a measurable (smooth)  coboundary. As time-changes defined by a measurably (smoothly)  cohomologous cocycles are measurably (smoothly)  isomorphic, similarly special flows over the same map under measurably (smoothly)  cohomologous roof functions are measurably (smoothly)  isomorphic (we refer for example to \cite{Katok:CC}).

Any time-change  $\{h^\alpha_t \}_{t \in \R}$ of $\{{h}_t \}_{t\in \R}$ determines 
the same return map $f:\Sigma \to \Sigma$, but a different return time function 
$\Phi^\alpha: \Sigma \to \R^+$. In fact, the following elementary result holds.

\begin{lemma} 
\label{lemma:returntimefs}
The return time function $\Phi^\alpha:\Sigma \to \R^+$ is given by the formula:
$$
\Phi^\alpha (x) = \int_0^{\Phi(x)}  (\alpha\circ h_t)(x)  dt \,, \quad \text{ for all } x\in \Sigma\,.
$$
\end{lemma}

It follows in particular from Lemma~\ref{lemma:returntimefs} that the return time functions $\Phi^\alpha$ and $\Phi$ are \emph{cohomologous} with respect to the return map $f:\Sigma \to \Sigma$ if and only if the function $\alpha:X \to \R$ is cohomologous to the constant function equal to $1$ for the infinitesimal generator of the flow $\{h_t(x)\}_{t\in \R}$. 

\smallskip
In the rest of this section we will consider the case when $\Sigma = \T^2$ and $f: \T^2\rightarrow \T^2$ is a \emph{linear skew-shift} over a circle rotation,
defined as  
\begin{equation}
\label{eq:skewshift}
 f(x,y):=(x+\alpha,y+x+ \beta) \,, \quad \text{ for all } \,(x,y) \in \T^2, \quad \mathrm{where} \, \alpha, \beta \in \R.
 \end{equation}
 We will also assume that $f$ is uniquely ergodic, which is equivalent to $\alpha \in \R \backslash \Q$ (see \cite{CFS:erg}). 
As we saw in~\S~\ref{sec:returnmaps} (see Lemma~\ref{lemma:returnmap}), any uniquely ergodic Heisenberg nilflow $\phi^W$ has a global cross section on which the first Poincar\'e map has the 
form~(\ref{eq:skewshift}). 
We will denote by $\Leb$ (respectively $\Leb^2$) the one-dimensional (respectively the two-dimensional) Lebesgue measure and by  $\mu$ be the probability measure obtained by normalization of the restriction of  the measure $\Leb^2 \times \Leb$ on $\T^2 \times \R$ to the domain of $f^\Phi$
(the normalizing factor is equal to $1/\int_{\T^2} \Phi(x,y) \ud x \ud y$).  By construction $\mu$ is invariant 
under the special flow $f^\Phi$ on $\Sigma/\sim_\Phi$.  

\smallskip It is well-known (see Lemma~\ref{lemma:wm}, \S~\ref{proof:mainthms}) that if the roof function $\Phi>0$ is a \emph{measurable (smooth)  almost coboundary}, that is, if there exists a measurable (smooth)  function 
$u:X \to \R$ such that 
$$
u \circ f - u = \Phi - \int_{\T^2} \Phi \ud \Leb\,,
$$
then the special flow $f^\Phi$ is measurably (smoothly)  isomorphic to a special flow with constant 
roof function over the skew-shift.  In this case, we will call the special flow $f^\Phi$ \emph{measurably (smoothly)  trivial}. Any measurably trivial special flow is not weak-mixing, hence not mixing (see again Lemma~\ref{lemma:wm}, \S~\ref{proof:mainthms}).

\smallskip
We will show that there is a  class  $\MM$ of smooth roof functions which correspond to smooth 
\emph{mixing} special flows over a uniquely ergodic skew-shift and that  $\MM$ is \emph{generic} 
in a  precise sense. In fact, we prove the following.

\begin{thm}[Mixing special flows]\label{thm:main}
There exists a  subspace $\TT \subset \RR \subset C^\infty(\T^2)$ of countable codimension 
in a dense subspace  $\RR \subset C^\infty(\T^2)$ such that for any positive function $\Phi \in \RR$ the 
following properties are equivalent:
\begin{enumerate}
\item  the roof function $\Phi \in \MM:= \RR \setminus \TT$;
\item the special flow $f^\Phi$ is not smoothly trivial;
\item  the special flow $f^\Phi$ is weak mixing;
\item the special flow $f^\Phi$ is mixing.
\end{enumerate}
\end{thm}
It is natural to ask whether Theorem~\ref{thm:main} generalizes to linear skew-shift on $\T^n$ with $n>2$. The implication $(1)\Rightarrow (4)$ could be proved for higher dimensional skew shifts, see Remark~\ref{higherskewproduct_rem} in~\S~\ref{growth_proof_sec}. On the other side, the implication  $(2)\Rightarrow (1)$ requires the analogue of the cocycle effectiveness (Theorem~\ref{thm:effectiveness} below) which relies on estimates currently known only for $n=2$ (see Remark~\ref{highereffectiveness_rem1} below and Remark~\ref{highereffectiveness_rem2} in~\S~\ref{sec:effectivenessproof}). 

\smallskip

Let us remark that the generic subset $\MM$ in Theorem~\ref{thm:main} is  concretely described in terms of invariant distributions (see \S~\ref{sec:effectivenessproof}). Thus, it is possible to check \emph{explicitely}  if a given smooth roof function given in terms of a Fourier expansion belongs to $\MM$ and to give concrete examples of mixing reparametrizations.

\smallskip
{\bf Examples.}

The following roof functions all give examples of  mixing special flows. 
\begin{enumerate}
\item $\Phi (x,y) =\sin ( 2 \pi y) + 2$;
\item $\Phi (x,y ) = \cos ( 2\pi (k x +  y)) + \sin(2\pi l x) + 3$, $k,l \in \Z$;
\item $\Phi(x,y) = \Re \sum_{j \in \Z} a_j e^{2\pi i(j x+ y) } + c$, if $\sum_{j \in \Z} a_j e^{-2\pi i (\beta j+ \alpha \binom{j}{2} )}   \neq  0$ and $c$ is such that $\Phi >0$.
\end{enumerate}
Example (1) shows that it is enough to have oscillations in the $y$-variable to produce mixing. 
 We show that the roofs in the examples above belong to the class $\MM$ at the end of  
 \S~\ref{sec:effectivenessproof}, after Corollary~\ref{complementkernels}.

\subsection{Mixing roof functions}\label{sec:mixingroofs}
Here we define the class of roof functions considered to obtain mixing special  flows.  
Let $\pi: \T^2 \to \T$ be the projection defined as $\pi(x,y) =x$ for all $(x,y) \in \T^2$.
The space $\pi^\ast L^2(\T) := \{ \Phi \circ \pi : \, \, \Phi \in L^2(\T) \}$ is a closed subspace
of $L^2(\T^2)$, hence there is an orthogonal decomposition
$$
L^2(\T^2) =  \pi^\ast L^2(\T) \oplus  \pi^\ast L^2(\T)^\perp\,.
$$
 We introduce the following notation for the orthogonal projections of a function
$\Phi \in L^2(\T^2)$ onto the components of the above splitting:
\begin{eqnarray}
\phi(x,y)&:=& \Phi(x,y)- \int \Phi(x, y ) dy \in \pi^\ast L^2(\T)^\perp , \label{phidef}\\
 \phi^{\bot}(x) &:=& \int \Phi(x,y) dy \in \pi^\ast L^2(\T) \equiv L^2(\T). \label{phiperpdef}
\end{eqnarray}

\begin{defn}[Roofs class $\RR$]  \label{polyn_class}
For an integer $d \geq 1$, let $\PP_d$ be the space of all continuous $\Phi$ such that for each 
$x\in \T$,  $\Phi(x, \cdot)$ is a trigonometric polynomial of degree at most $d$ on $\T$.  
Let $\PP:=\bigcup_{d \geq 1} \PP_d$.

\smallskip
The function $\Phi\in \RR $ if and only if 
$\Phi \in \PP$ and its projection
$ \phi^{\bot} $, defined in~(\ref{phiperpdef}),   
 is a trigonometric polynomial on $\T$. 
\end{defn} 
\noindent We remark that if $\Phi \in \RR$, we can write $\Phi(x,y)= \sum_{k=-d}^d c_k(x) e^{2\pi i k  y}$, 
since $\Phi \in \PP_d$, and $c_0(x)$ is a trigonometric polynomial.   
By definition the set $\RR\subset C^\infty(\T^2)$ is a dense subspace.

\begin{defn}[Trivial roofs $\TT$ and mixing roofs $\MM$]  \label{mixing_class} 
A function $\Phi$ belongs to $ \TT$ if and only if  $\Phi \in \RR$ 
and its projection $\phi$ defined in~(\ref{phidef}) is a measurable coboundary for the map $f:\T^2 \to \T^2$. Let us set $\MM := \RR\setminus \TT$, so that $\Phi$ belongs to $ \MM$ if and only if  $\Phi \in \RR$ 
and $\phi$ is \emph{not} a measurable coboundary.
\end{defn}

One of the two  main steps in the proof of Theorem~\ref{thm:main} is given by the following Theorem. 

\begin{thm}[Mixing]\label{mixing}
For any positive roof function $\Phi $   belonging to the class $\MM$ in Definition~\ref{mixing_class} the special flow $f^\Phi$  is mixing.
\end{thm}

The crucial ingredient in the proof of Theorem~\ref{mixing} is given by the a result on the
growth of Birkhoff sums of the skew-shift. 

\begin{thm}[Stretch of Birkhoff sums] \label{growth}
Assume that $\Phi \in \MM$, thus $\phi$ is not a measurable coboundary.
Then for each $C>1$, 
$$ \lim_{n\to \infty} \Leb(\vert \phi_n\vert <C) = 0\,.
$$ 
\end{thm}
\noindent The proof of Theorem~\ref{growth} is given in~\S~\ref{growth_proof_sec}, while the proof of 
Theorem~\ref{mixing} is in~\S~\ref{mixing_proof_sec}.

\subsection{Cocycle Effectiveness}\label{sec:effectiveness}
The following effectiveness result for coboundaries (in the sense of \cite{Katok:CC}, Def. 11.4) leads to a complete explicit description of the set $\MM$  in terms of Fourier series and it constitutes another main step in the proof Theorem~\ref{thm:main} (see \S~\ref{proof:mainthms}). 

We recall that, as found by A. Katok  \cite{Katok:CC}, \S 11.6.1, there are countably many 
independent obstructions (which are not signed measure)  to the existence of smooth solutions 
of the cohomological equation $u \circ f - u =\phi$. Such obstructions are invariant distributions 
for the skew-shift (see Theorem~\ref{thm:smoothcb}, \S~\ref{sec:effectivenessproof}).  If $\phi \in \pi^\ast L^2(\T)^\perp$ is smooth and belongs to the kernel of all $f$-invariant distributions, then the \emph{transfer function}, that is, the unique zero average solution 
of the cohomological equation, is smooth.

We will show that, if a sufficiently smooth function $\Phi$ such that 
$\phi^\perp=0$ is a coboundary for a skew-shift $f$ on $\T^2$ with a measurable 
transfer function, then the transfer function is smooth and $\Phi$ belongs to the
kernel of the (infinite dimensional) space of all $f$-invariant distributions.  More precisely, let $W^s(\T^2)$ denote the standard Sobolev space on $\T^2$, that is, the space
of all functions $\Phi = \sum_{(m,n)\in \Z^2} \Phi_{m,n} \exp( 2\pi i (mx +ny))$ such that
$$
\Vert \Phi \Vert_s :=\left( \sum_{(m,n)\in \Z^2}   (1+ m^2 +n^2)^s \vert \Phi_{m,n}\vert^2\right)^{1/2}
\, < \, +\infty \,. 
$$
\begin{thm}[Cocycle Effectiveness]
\label{thm:effectiveness}
Let $f$ be any uniquely ergodic skew-shift on $\T^2$ as in~(\ref{eq:skewshift}). For any function
$\phi \in \pi^\ast L^2(\T)^\perp \cap W^s(\T^2)$ for $s>3$  the following holds. If $\phi$ is a measurable coboundary, then it belongs  to the kernel of all $f$-invariant distributions and the transfer function $u \in W^t(\T^2)$ for all $t<s-1$.
\end{thm}
The proof of Theorem~\ref{thm:effectiveness}, given in~\S~\ref{sec:effectivenessproof}, is based on the quantitative estimates on equidistribution of nilflows  by L.~Flaminio and the second author in \cite{FlaFor:nil}.

\begin{rem}\label{highereffectiveness_rem1}
Theorem~\ref{thm:effectiveness} above answers a question posed by A.~Katok in \cite{Katok:CC},
\S 11.6.1, p.~88. For higher dimensional skew-shifts (or for any other higher dimensional 
nilpotent linear map) the analogous of Theorem~\ref{thm:effectiveness} is not known. 
\end{rem}

\section{Stretch of Birkhoff sums} \label{growth_proof_sec}
In this section we prove Theorem~\ref {growth}. Let $\Phi$ be any continuous function on $\T^2$ such that its projection $\phi$ (see~(\ref{phidef})) is  \emph{not a measurable coboundary}.

Since the map $f$ is uniquely ergodic  we can derive the following result by a standard Gottschalk-Hedlund technique.

\begin{lemma}\label{GottHed}
For each constant $C>1$ and for all $(x,y) \in \T^2$,
$$
\frac {1} {N} \# \{0 \leq n \leq N-1,\, : \phi_n (x,y)\vert  <C\}\xrightarrow{N\to \infty} 0.
$$
\end{lemma}

\begin{proof}

Let $\mu_{N,x,y}$ be a probability measure on $\T^2 \times \R$ with atoms of equal mass along
$(f^k(x,y),\phi_k(x,y))$, $0 \leq k \leq N-1$.  
It is enough to prove that  $\mu_{N,x,y} \to 0$ in the weak-$*$ topology, as $N \to \infty$, independently of $(x,y)\in \T^2$.
If this did not happen, we would be able to take a non-trivial limit, which would be a measure $\mu$ with non-zero mass, such that $F_* \mu=\mu$, where $F(x,y,z)=(f(x,y),z+\phi(x,y))$.

By unique ergodicity of $f$, $\pi_* \mu$ is a multiple of $\Leb$, where $\pi(x,y,z)=(x,y)$, and the conditional measures $\mu_{x,y}$ coincide up to translation: for almost every $x,y,x',y'$,
$\mu_{x,y}=T_* \mu_{x',y'}$ where $T(z)=z+t$, with $t=t(x,y,x',y')$.  By invariance, we have
$t(x,y,f(x',y'))=t(x,y,x',y')+\phi(x',y')$.
Choosing $(x_0,y_0)$ in a full measure set, and defining $u(x,y)=t(x_0,y_0,x,y)$, we get 
$\phi=u \circ f-u$, which contradicts the assumption that $\phi$ is not a measurable coboundary.
\end{proof}

\begin{cor} \label {density one}

For each constant $C>1$,
$$
\frac {1} {N} \sum_{n=0}^{N-1} \Leb(\vert \phi_n\vert <C) \xrightarrow{N\to \infty} 0.
$$

\end{cor}
\begin{proof}

The functions $\frac{1}{N} \sum_{n=0}^{N-1} \chi_{(-C,C)}\circ \phi_n $, where $ \chi_{(-C,C)}$ is the characteristic function of the interval $(-C,C) \subset \R$ converge pointwise to zero  by Lemma~\ref{GottHed}. Thus, the Corollary follows immediately  by integration over $\T^2$ and by the Lebesgue dominated convergence theorem .  
\end{proof}

\begin{lemma} \label {degree fixed}
 
For each $d \geq 1$ and for any norm $\Vert \cdot \Vert_d$ on $\C^{2 d}$, there exist constants 
$B_d>0$ and $b_d$ such that if $ {\bf c}= (c_{-d},\dots,c_{-1},c_1,\dots,c_d) \in
\C^{2 d}$ is a vector of unit norm
(that is, $\Vert {\bf c} \Vert_d=1$) then for every $\delta>0$ we have
$$
\Leb ( \vert \sum_{0<\vert k \vert \leq d} c_k e^{2 \pi i k x} \vert  <\delta )<B_d \delta^{b_d}.
$$
\end{lemma}

\begin{proof}
For fixed $d\geq 1$, the set  of trigonometric polynomials $\sum_{\vert k\vert \leq d} 
c_k e^{2 \pi i k x}$ with $\Vert {\bf c} \Vert_d=1$ forms a compact set of the space of  functions 
of class $C^{2d}$ on $\R$ with critical points of order at most $2d$, which gives the estimate.
\end{proof}

From now on we assume that $\Phi \in \PP_d$.

\begin{lemma}\label{decoupling}
 Let $C>1$. For any $\epsilon'>0$, there exist $C'>1$ and $\epsilon'' >0$ such that for all 
$n \geq 1$ such that $\Leb(\vert\phi_n\vert<C')<\epsilon''$, there exists $N_0:=N_0(C,\epsilon', n)\in \N$
such that for all $N\geq N_0$,  we have $\Leb(\vert\phi_N \circ f^n-\phi_N\vert<2 C)<\epsilon'$.
\end{lemma}

\begin{proof}

Indeed, let us write 
$$
\phi_n(x,y)= 2 \Re \sum_{0<k \leq d} c_{k,n}(x) e^{2 \pi i k y} = 
 \sum_{0<\vert k\vert \leq d} c_{k,n}(x) e^{2 \pi i k y}\,,
 $$ 
 with $c_{-k,n}(x) =\overline{c_{k,n}} (x)$ for all $0 < k \leq d$, $x \in \T$.  Then
$$
\phi_N \circ f^n(x,y)-\phi_N(x,y)=\phi_n \circ f^N(x,y)-\phi_n(x,y)=
 \sum_{0<\vert k\vert  \leq d} c_{k,N,n}(x) e^{2 \pi i k y}
$$
where we have denoted
$$ c_{k,N,n}(x):=e^{2 \pi i k [\binom{N}{2} \alpha +N \beta]}
{c_{k,n}(x+N  \alpha )}  e^{2 \pi i  k N x}  -  c_{k,n}(x). $$

Given any two complex numbers $c_i = \rho_i e^{ \theta_i}$, $i=1,2$,  for each $0\leq \theta < \pi/2$,  if $\theta_2 + 2\pi k N x \notin  ( \theta_1 + \pi - \theta,  \theta_1 + \pi + \theta ) + 2\pi \Z  $, then by elementary trigonometry $\vert c_2 e^{2\pi i k N x } - c_1\vert  \geq \vert c_1\vert\sin \theta$. 
Thus, for any interval $I \subset \T$ of length to most  $\delta>0$ and for $0\leq \theta < \pi/2$ we have
\bes
\Leb \{ x \in I \,\, : \,\,  \vert c_2 e^{2\pi i k N x } - c_1\vert \leq \vert c_1\vert  \sin \theta \} \leq \delta 
\frac{\theta}{\pi} + \frac{\delta}{ kN}.
\ees  
By uniform continuity of $c_{k,n}$, it is possible to choose $\delta >0$ so that if $\vert x-x'\vert \leq \delta$, 
then $\vert c_{k,n}(x)- c_{k,n}(x')\vert \leq 1/3$ and let us decompose $\T$ into intervals of size at most 
$\delta$. If $[x_1, x_2)$ is one of these intervals and $x \in [x_1, x_2]$, if we set 
$$
c_1:=c_{k,n}(x_1)\,, \quad c_2 := e^{2 \pi i k [\binom{N}{2}\alpha +N\beta] } {c_{k,n}(x_1+N  \alpha )}  
e^{2 \pi i  k N x_1} 
$$ 
and write 
$$
\vert c_{k,N,n}(x)\vert = \vert c_{k,N,n}(x_1) - (  c_{k,N,n}(x_1) - c_{k,N,n}(x)) \vert \geq  
\vert c_{k,N,n}(x_1) \vert - 2/3 \,,
$$ 
we can use the estimate above on each interval and get that, for every $0<k \leq d$ and 
for $0<\theta<\frac {\pi} {2}$, the following bound holds:
$$
\limsup_{N \to \infty} \Leb \left( \vert c_{k,N,n}(x)\vert  < \vert c_{k,n}(x)\vert \sin \theta - 2/3 \right)
 \leq \frac {\theta} {\pi}.
$$
By the the hypothesis on $n\in \N$ we have $\sum_{0<\vert k\vert  \leq d} \vert c_{k,n}(x)\vert  \geq C'$ except for a set of $x\in \T$ of Lebesgue measure $\epsilon''>0$. 
Recall that  $\vert c_{k,N,n}(x)\vert  = \vert c_{-k,N,n}(x)\vert $, so choosing $\theta \in  [0, \pi/2)$ 
such that $\sin \theta < 1/\sqrt{C'}$ and, using the fact that $\sin \theta>\frac {2} {\pi} \theta$ for all 
$0<\theta<\frac {\pi} {2}$, for $N$ sufficiently large we have, outside a set of measure $d (\theta / \pi) + 2\epsilon'' \leq  d   /2\sqrt{C'} + 2\epsilon'' $, the inequality
$$
\sum_{0<\vert k\vert  \leq d} \vert c_{k,N,n}(x)\vert  \geq \sqrt{C'} - 2/3 \geq \sqrt{C'}/3\,.
$$
By Lemma~\ref {degree fixed}, whenever $x\in \T$ is such that
$\sum_{0<\vert k\vert  \leq d} \vert c_{k,N,n}(x)\vert  \geq \sqrt{C'}/3$, we have that
$\vert  \sum_{0<\vert k\vert  \leq d} c_{k,N,n}(x) e^{2 \pi i k y}\vert  \geq 2 C$, except for a set of $y\in \T$ 
of Lebesgue measure at most $B_d \left (\frac {6 C} {\sqrt{C'}} \right )^{b_d}$.  Choose $\epsilon''>0$ 
and $C'>1$ be such that 
$$
B_d ( {6 C} / {\sqrt{C'}}  )^{b_d} + d/2\sqrt{C'} + 2\epsilon'' <\epsilon'\,.
$$
The result follows.
\end{proof}

\begin{proof} [ Proof of Theorem~\ref {growth}]
Let  $C>1$ and $\epsilon>0$ be fixed. We prove below that for every 
$N$ sufficiently large $\Leb(\vert \phi_N\vert <C)<\epsilon$ . 

Let us fix an integer $A \geq 1$ and $\epsilon'>0$ such that  $1/(A+1) + A (A+1) \epsilon'/2 <\epsilon$.  
By Lemma~\ref{decoupling} there exist  $C'>0$ and $\epsilon''>0$ 
 such that if $\Leb(\vert \phi_n\vert <C')<\epsilon''$ then $\Leb(\vert \phi_N \circ f^n-\phi_N\vert <2 C)<\epsilon'$
 for all $N\geq N_0(C,\epsilon,n)$. By Corollary~\ref {density one}, we can find $l \geq 1$ such that 
 for each $n=jl$ with $1 \leq j \leq A$ we have $\Leb(\vert \phi_n\vert <C')<\epsilon''$. Let
 $N_1:= \max\{ N_0(C,\epsilon,jl) : 1\leq j \leq  A)\}$. We claim that, for every $N\geq N_1$, 
 \begin{equation}
 \label{eq:progmeas}
  \Leb\left (\bigcup_{ 0\leq j<j'\leq A} \{ \vert \phi_N \circ f^{jl}-\phi_N \circ f^{j'l}\vert  < 2 C\} \right) \leq
  \frac {A (A+1)} {2} \epsilon'  \,.
 \end{equation}
 In fact, for every $0\leq j<j' \leq A $,  for $N\geq N_1 \geq N_0(C, \epsilon, (j'-j)l)$, we have
 $\Leb(\vert \phi_N \circ f^{(j'-j)l} -\phi_N\vert <2 C)<\epsilon'$, but since $f$ is measure preserving
 $$
 \Leb ( \vert  \phi_N \circ f^{jl}-\phi_N \circ f^{j'l}\vert  < 2 C) = \Leb (\vert \phi_N \circ f^{(j'-j)l} -\phi_N\vert <2 C) < \epsilon'\,.
 $$
 The claim follows as $\#\{(j,j') : 0 \leq j < j' \leq A\}$ is equal to $\frac {A (A+1)} {2} $.
 
By  construction, for all $N\geq N_1$, the sets $f^{-jl} \{\vert \phi_N\vert  < C\}$ are pairwise disjoint 
for $j=0, \dots, A$ outside a set of measure at most $\frac {A (A+1)} {2} \epsilon' $ (the set 
in formula~(\ref{eq:progmeas})), hence again by the measure preserving property of the map and choice of $A$ and $\epsilon'$, 
$$
\Leb( \vert \phi_N \vert  <C) \leq  \frac{1}{A+1} + \frac {A (A+1)} {2} \epsilon' < \epsilon\,. 
$$
The proof is complete. 
\end{proof}

\begin{rem}\label{higherskewproduct_rem} While Lemma~\ref{GottHed} holds  for any uniquely ergodic transformation,  Lemma~\ref{decoupling} exploits the parabolic divergence of orbits of the skew product in the neutral (isometric) direction. A similar result  could be proved more in general for higher dimensional skew-product maps of $\T^k$. In fact, it holds most likely for return maps of arbitrary uniquely ergodic nilflows. However, in this latter case, a more indirect argument is needed since exact formulas are not available in general. 
\end{rem}

\section{Mixing}\label{mixing_proof_sec}
In this section we give the proof of the main mixing result, Theorem~\ref{mixing}. We are going to use the following mixing criterium.
\subsection{Mixing criterium}
Let $\{f^\Phi_t\}_{t\in \R}$ the special flow over a uniquely ergodic skew-shift of the form
(\ref{eq:skewshift}) under the roof function $\Phi:\T^2 \to \R^+$. 
We recall the definition of the 
special flow. For all $t \in \R$,  let 
$$
n_t(x,y) :=  \max \{ n\in \N : \, \Phi_n (x,y) <t \}\,, \quad   \text{\rm for all } (x,y)\in \T^2\,.
$$
By the definition~(\ref{susp flow def}) of a special flow on a fundamental domain of 
$\T^2\times \R/ \sim_\Phi$, for any $(x,y)\in \T^2$ we have 
$$
f^\Phi_t ((x,y), 0)  =   (f^{n_t(x,y)},  t - \Phi_{n_t(x,y)}) \,.
$$
In order to show mixing, it is enough to prove the following. Let us call \emph{cube} any set of the form $[x_1,x_2]\times [y_1,y_2] \times [0,h]$ where $[x_1,x_2],  [y_1,y_2] \subset \T$ and $0< h< \min \Phi$. 
Let us call a \emph{partial partition into intervals} of $\{x\} \times \T$  any collection of intervals $I $ of the form  $I= \{ x \} \times [y',y'']$, $[y',y''] \subset \T$ with pairwise disjoint interiors. 

\begin{lemma}[Mixing Criterium]\label{mixingcriterium}
The flow $f^{\Phi}$ is mixing if, for any  cube $Q$, any $\epsilon>0$ and any $\delta >0$, there exists $t_0>0$ and for all $t\geq t_0$  
there exists a measurable set $X(t)\subset \T$ and for each $x \in X(t)$ there exists a partial partition 
$\xi(x,t)$ into intervals of $\{x\} \times \T$ such that
\be\label{measure partitions}
\Leb^2 \left(\T^2 \backslash \left( \cup_{x\in X(t)} \cup_{I \in \xi(x,t) } \ I \right) \right)\leq \delta.
\ee
and for all $x\in X(t) $ and all $[y',y''] \in \xi(x,t)$, 
\begin{equation} \label{mixing interval estimate}
\Leb (\{x\} \times [y',y''] \cap f^{\Phi}_{-t} (Q)) \geq (1-\epsilon) (y''-y') \mu(Q), 
\end{equation}
where $\Leb$ denotes here the Lebesgue measure on the fiber $\{ x \} \times \T$. 
\end{lemma}
 
\noindent The Lemma follows easily using Fubini theorem. Details can be found in \cite{Fa:ana, Ul:mix}.
We are going to prove that $f^{\Phi}$ is mixing by  constructing sets $X(t)$ and partial partitions $\xi(x,t)$ of the fibers $\{x\} \times \T$ with $x\in X(t)$ which satisfy the mixing estimate~(\ref{mixing interval estimate}).

\subsection{Mixing mechanism outline.}  
The main mechanism that we use to prove ~(\ref{mixing interval estimate})  is a phenomenon of \emph{stretching of ergodic sums} in the $z$-direction. Let us first give an heuristic explanation of this mechanism and an outline of the proof. We recall that this type of mechanism was used to produce mixing reparametrizations of flows over Liouvillean rotations on $\T^2$  by Fayad in \cite{Fa:ana} and to prove mixing in a class of area-preserving flows on the torus (by Sinai and Khanin in \cite{SK:mix}) and on higher genus surfaces (by the last author in \cite{Ul:mix}). 

Fix $x_0\in \T$ and let $I=\{x_0\} \times [a,b]$ be a subinterval of the $y$-fiber $\{x_0\}\times \T$. The \emph{stretch} of $\Phi_n$ on $ I $ is by definition the following quantity:
\bes
\Delta \Phi_{n} (I) := \max_{a\leq y \leq b} \Phi_n (x_0, y ) - \min_{a\leq y \leq b} \Phi_n (x_0, y ). 
\ees
We will show using Theorem~\ref{growth} that one can find, for all sufficiently large $t$, a set of intervals  $I = \{x\} \times [y',y'']$  whose union has large measure in $\T^2$ and which have large stretch $\Delta \Phi_{n} (I)$ for all times $n$ of the form $n_t(x,y)$ for some $(x,y) \in I$. As shown in the next section \S~\ref{number_iterations_sec}, large stretch implies that the variation of the number of discrete iterations $n_t(x,y)$ with $(x,y) \in I$ is large. Moreover, we will show that in this construction $y\mapsto n_t(x,y)$ is monotone on $[y',y'']$. If we  subdivide $I$ into intervals $I_i$ on which  $n_t(x,y)$ is constant, the image under $f^{\Phi}_t$ of each $I_i$ is a 1- dimensional curve $\gamma_i= f^{\Phi}_t (I_i)$ which goes from the base (i.e.~ the set $\T^2\times \{0\}$) to the roof (i.e.~the set $\{\left( x,y, \Phi(x,y )\right): (x,y) \in \T^2\}$). Since $f$ sends  $y$-fibers to $y$-fibers and preserves  distances within $y$-fibers, the projection of each curve $\gamma_i$  under the map $(x,y,z) \mapsto (x,y)$ is an interval in another $y$-fiber of the same length than $I_i$. If the intervals $I$ are chosen sufficiently small, the projections of the curves $\gamma_i$ shadow with good approximation an orbit of $f$. Moreover one can estimate the distortion of the curves  $\gamma_i$ and show that they are close to segments in the $z$-direction. Using that the skew-product $f$ is uniquely ergodic, together with estimates on the distortion, we can hence show that $f^\Phi_i (I) $ which is the union  of the curves $\gamma_i$  becomes equidistributed and hence prove the mixing estimate~(\ref{mixing interval estimate}).  

\subsection{Stretching and discrete number of iterations}\label{number_iterations_sec}
In the following sections we will denote by  $\overline{\Phi}$ and $\underline{\Phi}$ respectively the maximum and the minimum of  $\Phi$ on $\T^2$. By assumption  $\underline{\Phi}>0$.   
We will need later the following simple estimate on the variation of the discrete number of iterations $n_t(x,y)$ on a fiber interval $I = \{x\} \times [a,b]$ in terms of the stretch on $I$.   
\begin{lemma}\label{discrete iterations lemma}  
Let  $I=\{x\} \times [a,b]$. Let us denote by $\underline{n}_t(I) : = \min_{a\leq y \leq b} n_t(x,y)$ and by  $\overline{n}_t(I) : = \max_{a\leq y \leq b} n_t(x,y)$.  We have
\be\label{discrete_iterations_estimate}
\frac{\Delta \Phi_{\underline{n}_t(I)} (I)}{\overline{\Phi}} - \frac{\overline{\Phi}}{\underline{\Phi}} \leq \overline{n}_t(I) -\underline{n}_t(I) \leq \frac{ \Delta \Phi_{\underline{n}_t(I)} (I)   }{\underline{\Phi}} + \frac{\overline{\Phi}}{\underline{\Phi}}. 
\ee
\end{lemma}
Clearly~(\ref{discrete_iterations_estimate}) is meaningful when the stretch $\Delta \Phi_{\underline{n}_t(x,b)} (I)$ is large and hence shows that in this case also the variation of $n_t(x,y)$ on $I$ is large.  
\begin{proof}
Let us write for brevity of notation $\overline{n}_t := \overline{n}_t(I)$ and  $\underline{n}_t:=\underline{n}_t(I)$. 
Let $\underline{y}, \overline{y} \in [a,b]$ be such that respectively $n_{t}(x, \underline{y}) = \underline{n}_t$ and $n_{t}(x, \overline{y}) = \overline{n}_t$.  Writing  $  \Phi_{\overline{n}_t} (x,\overline{y})  =   \Phi_{\underline{n}_t} (x,\overline{y})  + \Phi_{ \overline{n}_t -\underline{n}_t  }\left( f^{\underline{n}_t}(x,\overline{y}) \right) $ and using the trivial estimate $\Phi_n \geq n \underline{\Phi}  $ we have
\be\label{upper estimate}
\begin{split}
 (\overline{n}_t -\underline{n}_t) \underline{\Phi} & \leq  \Phi_{\overline{n}_t -\underline{n}_t}\left(f^{\underline{n}_t}(x,\overline{y})\right)  =\\ & = \Phi_{\overline{n}_t} (x,\overline{y})  \pm  \Phi_{\underline{n}_t} (x,\underline{y})   -   \Phi_{\underline{n}_t} (x,\overline{y})  \leq t - (t - \overline{\Phi}) + \Delta \Phi_{\underline{n}_t} (I), 
\end{split}\ee
where the latter estimate uses that by definition $\Delta \Phi_{\underline{n}_t} (I) \geq \Phi_{\underline{n}_t} (x,\underline{y})   -   \Phi_{\underline{n}_t} (x,\overline{y})  $ and that since $n_{t}(x, \underline{y}) = \underline{n}_t$ and $n_{t}(x, \overline{y}) = \overline{n}_t$ we have  $t - \overline{\Phi} \leq  \Phi_{\underline{n}_t} (x,\underline{y}), \Phi_{\overline{n}_t} (x,\overline{y}) < t $. This proves the upper bound in~(\ref{discrete_iterations_estimate}). To prove the lower bound, let $y_n, {y}_M \in [a,b] $ such that  $  \Phi_{\underline{n}_t} (x,{y}_M) = \max_{a\leq y \leq b}   \Phi_{\underline{n}_t} (x,y) $ and  $  \Phi_{\underline{n}_t} (x,{y}_m) = \min_{a\leq y \leq b}   \Phi_{\underline{n}_t} (x,y) $. Reasoning as in~(\ref{upper estimate}) and remarking that $ \Phi_{\underline{n}_t} (x,{y}_M) \leq\Phi_{{n}_t(y_M)} (x,{y}_M) < t  $, we get
\bes
\begin{split}
 (\overline{n}_t -\underline{n}_t) \overline{\Phi} &\geq  (n_t(x,y_m) -\underline{n}_t) \overline{\Phi}  \geq  \Phi_{n_t(x,y_m) -\underline{n}_t}\left(f^{\underline{n}_t}(x,{y}_m)\right)  =\\ & = \Phi_{n_t(x,y_m)} (x,{y}_m)  \pm  \Phi_{\underline{n}_t} (x,{y}_M)   -   \Phi_{\underline{n}_t} (x,{y}_m)  \geq (t  - \overline{\Phi}) - t  + \Delta \Phi_{\underline{n}_t} (I).
\end{split}
\ees
\end{proof}

\subsection{From discrete time stretching to continuous time stretching.}\label{discrete to continuous}
 Theorem~\ref{growth} shows that $\vert \phi_n\vert $ (and hence the stretch on $y$-fibers) grows as $n\rightarrow \infty$. To prove that $f^{\Phi}$ is mixing, we need to show that the stretch grows as $t$ tends to infinity.  The following Lemma~\ref{growth_n_t(x)} is used to make this connection. In its proof we use the fact that the roof function belongs to the class $\RR$ introduced in Definition~\ref{polyn_class}.
 
We recall that, for a given $t> 0$, $n_t(x,y)$ is the maximum $n\in \N$ such that $\Phi_n (x,y) < t$.  For each $x\in\T$, let $\underline{n}_t(x) = \min_{y\in\T} n_t(x,y)$.  For each $C>0$, let 
\be \label{def X(t,C)}
 X(t,C) := \{ x  \ \mathrm{for}\ \mathrm{which} \ \mathrm{there}\ \mathrm{exists}  \ y_x \ \mathrm{such}\ \mathrm{that} \ \vert \phi_{\underline{n}_t(x)}(x,y_x) \vert >C \}. 
\ee
\begin{lemma}\label{growth_n_t(x)}
Let $\Phi \in \RR$. For each $C>1$, 
$\Leb \{ \T \backslash X(t,C) \} \rightarrow 0$ as $t\rightarrow \infty$.
\end{lemma}
\begin{proof}
Remark that if $x \notin X(t,C) $, then for all $ y\in \T$, $\vert \phi_{\underline{n}_t(x)} (x,y) \vert  \leq C$. Thus, if the conclusion of the Lemma does not hold, there exists $C>0$, $\delta>0$ and a subsequence $t_k  \rightarrow \infty$ as $k\rightarrow \infty$ such that for all $ k \in \N$
\be \label{badpoints}
\Leb^2 \{ (x,y) \in \T^2  : \vert \phi_{\underline{n}_{t_k}(x)} (x,y)  \vert  \leq C \} \geq 
\Leb  \{\T \backslash X(t,C ) \}
\geq \delta.  
\ee
Let us show  that in this case $n_{t_k}(x) $  as $x\in \T \backslash  X(t_k,C)$ assumes a finite number of values uniformly bounded in $k$. Since $\phi^{\bot}$ is a trigonometric polynomial by Definition 
\ref{polyn_class}, one can easily see using Fourier analysis that $\phi^{\bot} -\int \Phi(x,y) \ud x  \ud y$ is a coboundary, i.e.~there exists a $g$ such that  
$ \phi^{\bot}(x) =  g( x +\alpha ) - g(x) + \int \Phi(x,y) \ud x  \ud y$, and moreover $g$ is also a trigonometric polynomial.

For a fixed $t>0$,  let $\underline{y}(x)$ be such that $\underline{n}_{t}(x) = n_{t}(x, \underline{y}(x))$.  From the definition of special flow we have that $t -  \Phi ( f^{\underline{n}_{t}(x)-1} (x,\underline{y}(x) )) \leq \Phi_{\underline{n}_{t}(x)} (x, \underline{y}(x)) \leq t $. Moreover, by the decomposition $\Phi = \phi + \phi^{\bot}$, using that  $\phi^{\bot} -\int \Phi(x,y) \ud x  \ud y$ is a coboundary and $\int \Phi d x d y =1$, we have 
\bes
\Phi_{\underline{n}_{t}(x)} (x, \underline{y}(x))  =  \phi_{\underline{n}_{t}(x)} (x, \underline{y}(x)) + \underline{n}_{t}(x)  + g(x+\underline{n}_t(x) \alpha ) - g(x,y). 
\ees
So, denoting by $\overline{g} = \max g$ (well defined since here $g$ is a trigonometric polynomial and hence continuous) and by $\overline{\Phi} = \max \Phi(x,y)$, if $t=t_k$ for some $k$ and $x\notin X(t_k,C)$,  we have
$t_k - C - 2\overline{g} - \overline{\Phi} \leq   \underline{n}_{t_k} (x) \leq t_k +C  + 2\overline{g} $.  
This shows that there exists $N>0$ independent on $k$, such that $\left\vert  \underline{n}_{t_k}(x) - { t_k} \right\vert  \leq N$ for all $x \notin X(t_k,C)$. From this, recalling~(\ref{badpoints}), we can find for each $k$ some $n_k\in \N$ such that  $ n_k = \underline{n}_{t_k}(x)$ for some $x \notin X(t_k,C)$ and $
\Leb \{ (x,y) \in \T^2 :  \vert \phi_{n_k} (x,y) \vert  < C \} \geq \delta / N $. 
Since $\min_{x \in\T } n_{t_k}(x) \geq t_k / \overline{\Phi}$, $n_k \rightarrow \infty$ as $k\rightarrow \infty$ and this shows that $
\Leb \{ (x,y) \in \T^2 : \vert \phi_{n} (x,y) \vert  < C \} $ does not converges to zero as $n\rightarrow \infty$.   On the other side, $\phi$ is not a coboundary by Definition~\ref{mixing_class}, so
we got a contradiction with Theorem~\ref{growth}.
\end{proof}

\subsection{Choice of parameters.}\label{parameters}
Let $Q = [x_1,x_2]\times [y_1,y_2] \times [0,h]$ be a given cube. Given $\epsilon, \delta >0$, let us define the sets $X(t)$ and the partial partitions $\xi(x,t)$ which satisfy the conclusion in Lemma~\ref{mixingcriterium}. Let us first fix parameters $\delta_0,\epsilon_0, N_0, C_0,t_0$ as follows. The reader can skip these definitions at first (their use will become clear during the proofs).
\begin{enumerate}
\item \label{delta_0} Choose $0<\delta_0<1$ such that $(2d+1)\delta_0 + B_d \delta_0^{b_d} \leq \delta$, where $d$ is the degree of $y\mapsto \Phi(x,y)$ and $B_d$, $b_d$ are as in Lemma~\ref{degree fixed};

\item \label{epsilon_0} Choose $\epsilon_0>0$ such that $\epsilon_0 < \min\{\frac{\delta_0}{4d}, \underline{\Phi}, 1 \}$ and $(1-\epsilon_0)^5\leq (1-\epsilon)$; 

\item \label{N_0}
Let $\chi$ be a continuous function equal to $1$ on  $[x_1,x_2] \times [y_1,y_2-\epsilon_0 (y_2-y_1)]$ and identically $0$ outside $[x_1 ,x_2 ]\times [y_1 ,y_2- \frac{\epsilon_0 }{2}(y_2-y_1) ]$. Let us denote by $\Phi'(x,y):= \frac{\partial \Phi(x,y)}{\partial y } $ and by  $\Phi''(x,y):= \frac{\partial^2 \Phi(x,y)}{\partial y ^2 } $ and let us remark that $\Phi'$ and $\Phi''$   have zero average on $\T^2$ while $\Phi$ has 
integral equal to $1$.  

Since $f$ is uniquely ergodic and ergodic sums of continuous functions over an uniquely ergodic transformation converge uniformly (see \cite{CFS:erg}), there exists $N_0 \in \N$ be such that for all $n\geq N_0$ and for all $(x,y )\in \T^2$ all the following bounds hold simultaneously:
\begin{enumerate}
\item \label{BS Phi'} $\vert \Phi'_n (x,y)\vert \leq \epsilon_0 n$;
\item \label{BS Phi''}$\vert \Phi''_n (x,y)\vert \leq \epsilon_0 n $;
\item \label{BS Phi} $\vert \frac{\Phi_n (x,y)}{n}- 1\vert  \leq \frac{\epsilon_0}{1+\epsilon_0}$;
\item \label{BS chi}$\left\vert  \frac{ \chi_n (x,y) }{n} - \int_{\T^2} \chi  \right\vert  < \epsilon_0 $;
\end{enumerate}

\item \label{C_0}
Let $\overline{\vert \Phi'\vert }$ and $\overline{\vert \Phi''\vert }$ denote respectively the maximum of $\vert \Phi'\vert $ and 
$\vert \Phi''\vert $ on $\T^2$ and choose
$$
C_0 > \max \left\{ \frac{2d }{\delta_0^2}\frac{ \overline{\Phi}}{ \underline{\Phi} },   \frac{d N_0}{\delta_0} \max \{  \overline{\vert  \Phi'\vert } ,\overline{\vert \Phi''\vert } \}   ,   \frac{d(N_0+1)^2 }{ \delta_0^2}, \frac{d^2\pi^2}{\epsilon_0^{10}} \frac{\max \{ \overline{\vert \Phi'\vert } , 1 \} }{ (y_2-y_1)^2} \right\}.
$$ 
\item \label{t_0} Let $ X(t,C_0)$ be as in~(\ref{def X(t,C)}) in~\S~\ref{discrete to continuous}. By Theorem~\ref{growth_n_t(x)}, we can choose $t_0$ such that  for each $t\geq t_0$ we have  $\Leb\left(\T\backslash X(t,C_0) \right) <\delta_0$. 
\end{enumerate}

\subsection{Definition of $X(t)$ and preliminary $y$-fibers partitions.}
Fix any $t\geq t_0$ where $t_0$ is as in~(\ref{t_0}) in~\S~\ref{parameters}. Let us set $X(t):X(t,C_0)$ for $C_0$ defined in~(\ref{C_0}) in~\S~\ref{parameters}. For any $x\in X(t)$, let us define a preliminary partitions into intervals   $\xi_1(x,t)$ which we will later refine to obtain $\xi(x,t)$ with the properties in Lemma~\ref{mixingcriterium}. 

Let us write 
$$
\phi_{\underline{n}_t(x)} (x,y) = \Re \sum_{k=1}^d c_k(x) e^{2\pi i k y}\quad \text{and} \quad
\frac{\partial}{\partial y}\phi_{\underline{n}_t(x)} (x,y) = \Re \sum_{k=1}^d c'_k(x) e^{2\pi i k y},
$$
where $c'_k(x)=2\pi i k c_k(x) $. Let  $\delta_0$ be as in~(\ref{delta_0}) in~\S~\ref{parameters} 
and let us define
\be\label{xi0 def}
\xi_0 (x,t):= \left\{ y\in \T :  \vert  \phi'_{\underline{n}_t(x)} (x,y)\vert    \geq \delta_0 \max_{k}\vert c'_k(x)\vert   \right\} . 
\ee
Clearly $\xi_0(x,t)$ is a union of intervals.  Let $\xi_1(x,t)$ be the partial partition obtained by discarding from $\xi_0(x,t)$ all intervals which have length less than $\delta_0$. 

\smallskip
For brevity, we will denote in the following sections  
$$\phi'(x,y): ={ \partial{\phi(x,y)}}/{\partial y} \quad \text{ and } \quad \phi''(x,y): ={ \partial^2{\phi(x,y)}}/{\partial y^2}.
$$ 
\begin{lemma}\label{same derivative}
The following identities hold:
 $$\frac{\partial}{\partial y } \Phi_n(x,y)= \Phi_n'(x,y) = \phi_n'(x,y) \quad
 \text{and} \quad \frac{\partial^2}{\partial y^2} \Phi_n(x,y)=  \Phi_n''(x,y) = \phi_n''(x,y).
 $$
\end{lemma}
\begin{proof} The identity  $\frac{\partial}{\partial y } ( \Phi_n ) =( \frac{\partial}{\partial y } \Phi )_n $ holds since $f$ is a skew-product. In fact, any skew-product commutes with all translations in the $y$ coordinate, hence it commutes with the derivative $\partial/\partial y$.  By definition, the function
$$
\Phi_n(x,y) - \phi_n(x,y) =\sum_{k=0}^{n-1} \int \Phi(x +k\alpha,y) \ud y 
$$
depends only on $x\in \T$, hence $\phi_n(x,y_1)- \phi_n(x,y_2)= \Phi_n(x,y_1) - \Phi_n(x,y_2)$, for any $n\in \N$ and for any $x,y_1,y_2 \in \T$. The Lemma follows. 
\end{proof}

The following Lemma shows that on points belonging to  intervals in $\xi_1(x,t)$ both  the derivative $\Phi'_{\underline{n}_t(x)}$ and the stretch is large and of the same order. 

\begin{lemma} \label{xi1}
The partial partitions $\xi_1(x,t)$, $x\in X(t)$, are such that 
\be\label{measure}
\Leb^2 \left( \T^2 \backslash  \left( \cup_{x\in X(t)}  \cup_{I \in \xi_1(x,t) } \, I \right) \right) \leq \delta.
\ee
and for all $x\in X(t)$ 
\begin{eqnarray}
& & 
 \vert \Phi'_{\underline{n}_t(x)} (x,y) \vert   \geq \frac{2\pi \delta_0 }{d}C_0 , \quad \mathrm{for} \ \mathrm{all} \ y \in \xi_1(x,t); \label{big derivative}  \\
& & 
 \vert \Phi'_{\underline{n}_t(x)} (x,y_1) \vert   \geq \frac{\delta_0} {d}   \vert \Phi'_{\underline{n}_t(x)} (x,y_2) \vert  ,  \quad \mathrm{for} \ \mathrm{all} \ y_1, y_2 \in \xi_1(x,t).\label{ratios derivative}
\end{eqnarray}
Moreover,  for each   $I= \{x\} \times [a,b]$ with $[a,b]\in \xi_1(x,t)$, we have
\begin{eqnarray}
&&  \label{stretch estimates}
\Delta \Phi_{\underline{n}_t(x)} (I) \geq \frac{2\pi  \delta_0^2 }{d} C_0 , \\
&& \label{ratios stretch} 
\frac{\delta_0^2}{d} \max_{ y \in \xi_1(x,t)} \vert \Phi'_{\underline{n}_t(x)} (x,y) \vert  \leq 
 \Delta \Phi_{\underline{n}_t(x)} (I) \leq \frac{d}{\delta_0} \min_{ y \in \xi_1(x,t)} \vert \Phi'_{\underline{n}_t(x)} (x,y) \vert  , \\ 
&& \max_{a\leq y \leq b} \vert \Phi''_{\underline{n}_t(x)}(x,y)\vert  \leq \frac{2 \pi d^2 }{\delta_0^2} \Delta \Phi_{\underline{n}_t(x)} (I).\label{second derivative}
\end{eqnarray}
\end{lemma}
  
\begin{proof}
Let us first show the estimate on the total measure~(\ref{measure}). By  applying Lemma~\ref{degree fixed} to  $\phi'_{\underline{n}_t(x)} / \max_{k}c'_k(x)$,  we have $\Leb(\T\backslash \xi_0(x,t) )\leq B_d \delta_0^{b_d}$. Since the solutions of $\phi'_{\underline{n}_t(x)} / \max_{k}c'_k(x)  = \pm \delta_0$ are a subset of the zeros of a polynomials of degree at most $2d$, there are at most $2d$ solutions for each level set. Thus, $\xi_1(x,t)$ is obtained by removing at most $2d$ intervals of length smaller than 
$\delta_0$ from $\xi_0(x,t)$ and for each $x\in X(t) $ we have $\Leb ( \xi_1(x,t) ) \geq 1-  B_d \delta_0^{b_d} - 2d \delta_0$. Applying Fubini and recalling that $\Leb(X(t))\geq 1-\delta_0$ by~(\ref{t_0}) in 
\S~\ref{parameters}, we get~(\ref{measure}) by the choice~(\ref{delta_0}) of $\delta_0$ in 
\S~\ref{parameters}. 

Clearly for all $(x,y)$ we have $\vert \phi'_{\underline{n}_t(x)} (x,y)\vert     \leq d \max_{k}\vert c'_k(x)\vert  $. Thus, from the definition~(\ref{xi0 def}) of $\xi_0(x,t)$, we immediately have 
$$
\min\vert \phi'_{\underline{n}_t(x)} (x,y)\vert  \geq \delta_0 \max\vert \phi'_{\underline{n}_t(x)} (x,y)\vert  /d   \,, 
$$  
hence~(\ref{ratios derivative}) by Lemma~\ref{same derivative}. Moreover, since by definition of $X(t)$, there exists $y(x)$ such that $\vert \phi_{\underline{n}_t(x)} (x,y(x))\vert  \geq C_0$, we also have $\max_{k}\vert c_k(x)\vert  \geq C_0/d $ and since  $\max_k \vert c'_k(x)\vert \geq 2\pi \max_k \vert c_k\vert  \geq 2\pi C_0/d  $, again from the definition of $\xi_0(x,t)$ we get $\vert \Phi'_{\underline{n}_t(x)}(x,y)\vert \geq 2\pi \delta_0 C_0/d $, concluding the proof of~(\ref{big derivative}). 

The estimates~(\ref{stretch estimates},~\ref{ratios stretch}) on the stretch follows simply by using mean value  from (\ref{big derivative}) and~(\ref{ratios derivative}) respectively and from the lower estimate on the size of intervals in $\xi_1(x,t)$. The last estimate~(\ref{second derivative}) is obtained combining the trivial upper estimate $\vert \Phi''(x,y)\vert \leq 2\pi d^2 \max{\vert c_k'(x)\vert }$ with $\Delta \Phi_{\underline{n}_t(x)} (I) \geq \delta_0^2 \max_k \vert c'_k(x)\vert   $ which follows from mean value and definition of $\xi_0(x,t)$ and $\xi_1(x,t)$. 
 \end{proof}
 
Let us denote by $\overline{n}_t(x) = \max_{y \in \T} n_t(x,y)$.  The choices of parameters in~\S~\ref{parameters} and Lemma~\ref{xi1} guarantee that not only derivatives and stretch are large for $n=\underline{n}_t(x)$, but also remain large for  all further iterates up to $\overline{n}_t(x) $, as stated 
below.

\begin{lemma}\label{large all involved n}
For all $x\in X(t)$ and  $I=\{x\} \times [a,b]$ with $[a,b] \in \xi_1(x,t)$ the sign of $\Phi'_{\underline{n}_t(x)}$  on $I$ for  all $\underline{n}_t (x) \leq n \leq \overline{n}_t(x)$ is the same and   for  all $\underline{n}_t (x) \leq n \leq \overline{n}_t(x)$ we have:  
\begin{eqnarray}
&&  \frac{ \vert \Phi'_{\underline{n}_t(x)}(x,y) \vert }{2} \leq  \vert \Phi'_{n} (x,y)\vert  \leq \frac{3 \vert \Phi'_{\underline{n}_t(x)} (x,y) \vert  }{2}, 
 \quad \mathrm{for }\ \mathrm{all}\ y\in[a,b] ;\label{all large derivative} \\
&& \frac{1}{2}\frac{\delta_0}{d} \, \Delta \Phi_{\underline{n}_t(x)}(I)    \leq  \Delta \Phi_{n} (I) \leq \frac{3}{2}\frac{d}{\delta_0}  \Delta \Phi_{\underline{n}_t(x)}(I) ;
\label{all large stretch} \\
&& \vert \Phi''_n (x,y)\vert \leq \frac{4\pi d^2}{ \delta_0^2}\, \Delta\Phi_{\underline{n}_t(x)}(I),\quad \mathrm{for }\ \mathrm{all}\ y\in[a,b]\  .
\label{all controlled second} 
\end{eqnarray}
\end{lemma}

\begin{proof}
Fix $x \in X(t)$ and $n$ with $\underline{n}_t (x) \leq n \leq \overline{n}_t(x)$.  
We have $\Phi'_n (x,y)= $  $\Phi'_{\underline{n}_t(x)}(x,y) + $  $ \Phi'_{n- \underline{n}_t(x)} (x_t,y_t)$ where $(x_t,y_t):= f^{\underline{n}_t(x)}(x,y)$.  Let us show that $\vert \Phi'_{n- \underline{n}_t(x)} (x_t,y_t) \vert \leq \vert  \Phi'_{\underline{n}_t(x)} (x,y)\vert /2 $, from which it follows that~(\ref{all large derivative}) holds  and also that $\Phi'_{{n}}(x,y)$ has  constant  sign for all  $\underline{n}_t (x) \leq n \leq \overline{n}_t(x)$ and $y \in [a,b]$ (recall that $\Phi'_{\underline{n}_t(x)}$ has no zeros on $I$ by construction). 

Consider $N_0$ defined in~(\ref{N_0}) in~\S~\ref{parameters}.  On one hand, if $n- \underline{n}_t(x) \leq N_0$, we have
$$
\vert \Phi'_{n- \underline{n}_t(x)} \left( x_t,y_t\right) \vert \leq N_0 \vert  \overline{\Phi' } \vert  \leq \pi \delta_0 C_0 /d \,,
$$ 
by choice of $C_0$ in~(\ref{C_0}), \S~\ref{parameters}. Thus,  $\vert \Phi'_{n- \underline{n}_t(x)} (x_t,y_t) \vert  \leq \vert \Phi'_{ \underline{n}_t(x)} (x,y)\vert  / 2$ by~(\ref{big derivative}). 
On the other hand, if $n- \underline{n}_t(x) \geq N_0$, by~(\ref{BS Phi'}) in~\S~\ref{parameters}, then by Lemma~\ref{discrete iterations lemma} and then by mean value,~(\ref{ratios derivative}) and $\epsilon_0 \leq 1$, we get
\be\label{estimatePhi'}
\begin{aligned}
 \vert \Phi'_{n- \underline{n}_t(x)} (x_t,y_t) \vert  &\leq \epsilon_0 (  n- \underline{n}_t(x))  \\ \leq &\epsilon_0 \frac{ \Delta \Phi_{\underline{n}_t(x)} (\{x\}\times \T) + \overline{\Phi}}{\underline{\Phi}} 
 \leq  \epsilon_0 \frac{d}{\delta_0} {\vert \Phi'_{\underline{n}_t(x)} (x,y)\vert  }  + \frac{\overline{\Phi}}{\underline{\Phi}}. 
 \end{aligned}
\ee
The two terms in the last expression are both less than $\vert \Phi'_{\underline{n}_t(x)} (x,y)\vert  / 4 $, the first by choice of   $\epsilon_0$ in~(\ref{epsilon_0})  in~\S~\ref{parameters} and the second using~(\ref{big derivative}) and  $\pi \delta_0 C_0 /2d \geq \overline{\Phi}/\underline{\Phi} $, which follows by choice of $C_0$ in~(\ref{C_0})  in~\S~\ref{parameters}. This concludes the proof of~(\ref{all large derivative}). 

The estimate~(\ref{all large stretch}) follows from~(\ref{all large derivative}) by mean value and~(\ref{ratios derivative}). To get~(\ref{all controlled second}), separating as before the cases $n- \underline{n}_t(x) \leq N_0$ and $n- \underline{n}_t(x) \geq N_0$ and, in the second case, using~(\ref{BS Phi''}) in~\S~\ref{parameters} and reasoning as in the proof of~(\ref{estimatePhi'}), we have
\bes \vert \Phi''_{n} (x,y) \vert \leq  \vert \Phi''_{ \underline{n}_t} (x,y) \vert  + \max \left\{ N_0 \overline{\vert \Phi''\vert  },  
 \frac{\epsilon_0}{\underline{\Phi}}  \Delta \Phi_{\underline{n}_t(x)} (\{x\}\times \T) +  \frac{\overline{\Phi}}{\underline{\Phi}}  \right\}. 
\ees
The final estimate~(\ref{all controlled second}) follows from here estimating $\vert \Phi''_{ \underline{n}_t} (x,y) \vert $ by~(\ref{second derivative}) and  controlling the first term in the maximum  by using mean value,~(\ref{ratios stretch}) and $\epsilon_0  \leq  \underline{\Phi} \leq \pi d \underline{\Phi}$ (recall the choice of $\epsilon_0$ in~(\ref{epsilon_0}) in~\S~\ref{parameters}) and estimating the second term in the maximum by ~(\ref{stretch estimates}) and the choice of $C_0$ in~(\ref{C_0})  in~\S~\ref{parameters}, which guarantees that $2\pi d^2 \Delta \Phi_{\underline{n}_t(x)} /\delta_0^2 \geq 4\pi^2 d C_0 \geq \overline{\Phi}/\underline{\Phi}$.

\end{proof}

\subsection{Fibers Partitions}
Let us refine the partitions $\xi_1(x,t)$ so that we can prove that each interval equidistributes. 
Let us fix $x\in X(t)$ and $I=\{x\}\times [a,b] $ with $[a,b] \in \xi_1(x,t)$. Let us recall that we denote by 
$\underline{n}_t(I) = \min_{a\leq y\leq b  }  n_t(x,y) $ and by $\overline{n}_t(I) =\max_{a\leq y\leq b  }  n_t(x,y) $ and let $ \Delta n_t(I) := {\overline{n}_t(I) -\underline{n}_t(I)+1} $.  
The previous construction guarantees the following properties.
\begin{lemma}\label{monotonicity and variation}
For each $x\in X(t)$ and each $[a,b] \in \xi_1(x,t)$ the function 
$y \mapsto n_t(x,y)$ is monotone on $[a,b]$ and  $ \Delta {n}_t(I)  \geq {\pi \delta_0^2 C_0}/{d \underline{\Phi}} $. 
\end{lemma}
\begin{proof}
By Lemma~\ref{large all involved n}, the sign of $\Phi'_n$ on $I:= \{x\}\times [a,b]$ is the same for all  $ \underline{n}_t(I) \leq  n  \leq \overline{n}_t(I) $. Let us assume that  $\Phi'_ {\underline{n}_t(I)} <0$ so that $\Phi_n$ is monotonically decreasing on  $I$ for all  $ \underline{n}_t(I) \leq  n  \leq \overline{n}_t(I) $ and let us show that this implies that $y \mapsto n_t(x,y)$ is increasing on $[a,b]$. If $a\leq y_1 < y_2 \leq b$, we have $\Phi_{ n_t(x,y_1)} (x , y_2) < \Phi_{ n_t(x,y_1)} (x , y_1)\leq  t $ by definition of $n_t(x,y_1)$. Thus, by definition of $n_t(x,y)$ this shows $n_t(x,y_2) \geq n_t(x,y_1)$.  
From Lemma~\ref{discrete iterations lemma},  $\overline{n}_t(I)  -\underline{n}_t(I) \geq \Delta \Phi_{\underline{n}_t(I)} (I) /{\underline{\Phi}} -  \overline{\Phi}/{\underline{\Phi}}$ and  by~(\ref{stretch estimates}) we have $\Delta \Phi_{\underline{n}_t(I)} (I) /{\underline{\Phi}}  \geq {2\pi \delta_0^2 C_0}/{d \underline{\Phi}}$. This gives the desired estimate for $ \Delta n_t(I)$ since $ \overline{\Phi}/{\underline{\Phi}} \leq {\pi \delta_0^2 C_0}/{d \underline{\Phi}}$ by choice of $C_0$ in~(\ref{C_0}) \S~\ref{parameters}.
\end{proof}
From Lemma~\ref{monotonicity and variation}, we know that $I$ can be subdivided into exactly  $ \Delta n_t(I)$ maximal  intervals on which $y\mapsto n_t(x,y)$ is locally constant. Let us assume without loss of generality that $\Phi'_ {\underline{n}_t(I)} <0$. In this case,  more precisely,  for each $1 \leq j \leq \overline{n}_t(I) -\underline{n}_t(I)  $ there is a unique $y_j \in [a,b]$ such that $\Phi_{\underline{n}_t(I) + j}(x,y_j) = t$ and moreover $y_j < y_{j+1}$ for all $1\leq j \leq  \overline{n}_t(I) -\underline{n}_t(I) $. Thus, setting  $y_0:= a$ and $ y_{ \overline{n}_t(I) -\underline{n}_t(I)+1 }:= b $, for each $0 \leq j \leq \overline{n}_t(I) -\underline{n}_t(I)  $  the interval $(y_j,y_{j+1})$ is the interior of the maximal interval on which $n_t(x,y) $ is equal to $\underline{n}_t (I) +j $.

Let $ N_t(I) : =[ \sqrt{ \Delta n_t (I)} ]$, where $[z ]$ denotes the integer part of $z$. 
Let us group the intervals  $[y_j,y_{j+1}]$ into  $ N_t(I)$ groups, each of the first  $ N_t(I)-1$ made by exactly  $ N_t(I)$ consecutive intervals, the last by the remaining ones, which are at most  $2  N_t(I)$. In this way we obtain a subdivision of the interval $I$ of the partition $\xi_1(x,t)$ into intervals of the form $[y_{k  N_t(I) }, y_{(k+1)N_t(I)-1} ]$ for $k=0, \dots, N_t(I)-2$ or, in the case of the last interval, of the form  $[y_{N_t(I) (N_t(I)-1) }, b ]$.   

Let $\xi(x,t)$ be the partition obtained refining $\xi_1(x,t)$ by repeating the above subdivision for each interval $I \in \xi_1(x,t)$.
The elements $J \in \xi(x,t)$  have the following properties, used in the following  \S~\ref{area estimates sec}  to prove equidistribution~(\ref{mixing interval estimate}).

\begin{lemma}[Properties of fiber partitions]\label{properties}
For each interval $J=\{x\} \times [y',y'']$ with $x \in X(t)$ and $[y',y''] \in \xi(x,t)$ the following properties hold.  
\begin{eqnarray}
&& \label{asymptotics curves}
\left\vert  \frac{ \Delta n_t(J)  }{\Delta\Phi_{\underline{n}_t(J)}(J) }  -1 \right\vert  \leq \epsilon_0, \quad \mathrm{where} \quad \Delta n_t(J) = \overline{n}_y(J)  -  \underline{n}_y(J) +1 ;   
\\ 
&& \label{equidistribution base}
 \frac{1}{\Delta n_t(J)}  \sum_{n= \underline{n}_y(J)}^{\overline{n}_y(J) } \chi
 \left( f^{n}(x,y ) \right)  \geq   (1-\epsilon_0)^2 (x_2-x_1)(y_2-y_1) , \quad \forall \, y \in [y',y''];
 \\
&& \label{size J}  
\Leb(J) = \vert y''-y'\vert  \leq  \min \left\{ \frac{\epsilon_0 (y_2-y_1)}{2} , \frac{\epsilon_0}{2 \overline{\vert \Phi'\vert }} , \frac{\delta_0^3 }{8\pi d^3} \epsilon_0  \right\};
\\
&& \label{variation slopes}
\left\vert  \frac{\Delta\Phi_{\underline{n}_t(J)}(J) }{\Delta\Phi_{n }(J)  } -1 \right\vert      \leq  \epsilon_0, \quad n= \underline{n}_t(J),\dots,  \overline{n}_t(J) ;   
\end{eqnarray}
Moreover, if for $h>0$ and  $ \underline{n}_t(J) \leq n \leq \overline{n}_t(J)$  we denote by
\be \label{J_n^h def}
J_{n}^{h} :=  \{ y \in [y',y''] :  t - h  \leq \Phi_{n}(x,y) <  t  \}.
\ee
we also have
\be \label{distortion}  \left\vert \frac{\Delta \Phi_{n } (J )  \Leb(J^{h}_j) }{  (y''-y') h } -1  \right\vert  \leq  \epsilon_0 , \qquad n= \underline{n}_t(J),\dots,  \overline{n}_t(J) .
\ee
\end{lemma}

\begin{proof}
Let $I = \{x\} \times [a,b]$ with  $[a,b]\in \xi_1(x,t)$ be such that $J \subset I$. We will assume that $J$ does not contain neither of the endpoints of  $I$. The proofs in the latter case requires easy adjustments to take care of the intervals where $n_t(x,y) = \underline{n}_t(I)$ or  $ \overline{n}_t(I)$, which we leave to the reader. Let us remark that in this case  the values assumed by $n_t(x,y)$ on $J$, which by definition are $\Delta n_t(J)$ are by construction  exactly  equal to $ N_t(I)$. 

Without loss of generality, let us assume that $y \mapsto n_t(x,y)$ is increasing on $[y',y'']$. Thus, 
$\underline{n}_t(J)= n_t(x,y')$ and $\overline{n}_t(J)= n_t(x,y'')$, so  we have  $\Phi_{\underline{n}_t(J)} (x,y')= t = \Phi_{\overline{n}_t(J) +1 } (x,y'') $. Moreover,  in this case $\Phi_{\underline{n}_t(J)}$ is decreasing. Thus, 
\be
\begin{aligned}
 \Delta\Phi_{\underline{n}_t(J)}(J)  &= \Phi_{\underline{n}_t(J)} (x,y') -  \Phi_{\underline{n}_t(J)} (x,y'') \\
 &= \Phi_{\overline{n}_t(J)+1} (x,y'') -  \Phi_{\underline{n}_t(J)} (x,y'') = \Phi_{\Delta n_t(J)} 
 (f^{\underline{n}_t(J)} (x,y'')).
\end{aligned}
\ee
This shows that~(\ref{asymptotics curves}) follows from~(\ref{BS Phi}) in~\S~\ref{parameters}, which can be applied  since $\Delta n_t(J) = N_t(I)  \geq \sqrt{\Delta n_t(I)} -1 \geq N_0 $  by Lemma~\ref{monotonicity and variation} and by the inequality $\pi\delta_0^2 C_0/d\underline{\Phi}\geq N_0+1$, which holds by choice of $C_0$ in condition~(\ref{C_0}) in~\S~\ref{parameters}. For the same reason, condition~(\ref{BS chi}) in~\S~\ref{parameters} also holds and gives~(\ref{equidistribution base}) by remarking that, by definition of $\chi$  (see~(\ref{N_0}),  \S~\ref{parameters}),
$$
\int \chi - \epsilon_0 \geq \int \chi (1 - \epsilon_0) \geq (x_2-x_1)(y_2 -y_1)(1-\epsilon_0)^2\,.
$$ 
To estimate the size of $J$, let us remark that by mean value 
$$
\Delta \Phi_{\underline{n}_t(J)}(J) = \vert y''-y'\vert  \vert  \Phi'_{\underline{n}_t(J)} (x,\tilde{y}) \vert  \,,
\quad \text{ for some } \tilde{y} \in [y',y'']\,.
$$
 Since $\Delta \Phi_{\underline{n}_t(J)}(J) \leq 2N_t(I)$ by~(\ref{asymptotics curves}) and $ \vert \Phi'_{\underline{n}_t(J)} (x,\tilde{y})\vert  \geq \delta_0 \Delta \Phi_{\underline{n}_t(I)}(I) /2d$ by~(\ref{all large derivative}) and~(\ref{ratios stretch}), we get $\vert y''-y'\vert  \leq \frac{4d}{\delta_0} \frac{ N_t(I)}{\Delta\Phi_{\underline{n}_t(I)}(I) } $. By Lemma~\ref{discrete iterations lemma} and definition of $N_t(I)$, we have $\Delta\Phi_{\underline{n}_t(I)} \geq  N_t(I)^2 \underline{\Phi}  - \overline{\Phi} - \underline{\Phi}$. Thus, since $N_t(I)\geq \delta_0 \sqrt{C_0}/\sqrt{d}$  by Lemma~\ref{monotonicity and variation}, we have $\vert y'-y'' \vert \leq 4d /(\delta_0 N_t(I)) \leq 4d \sqrt{d}/ \sqrt{C_0 \delta_0^2} $. From here, one can check that by choice of  of $C_0$ and $\epsilon_0$ in~(\ref{epsilon_0},~\ref{C_0}) in  \S~\ref{parameters}, we have  $\vert y''-y'\vert  \leq \epsilon_0^4 (y_2-y_1)/\pi \max \{ \overline{\vert \Phi'\vert }, 1 \}$ and thus $\vert y''-y'\vert  $ satisfies~(\ref{size J}).
 
To estimate~(\ref{variation slopes}), using the definition of stretch, the cocycle properties of Birkhoff 
sums and then mean value, we can write  $\vert {\Delta\Phi_{\underline{n}_t(J)}(J) }- {\Delta\Phi_{n }(J)  }\vert $ \ 
$ \leq \vert \Phi'_{n- \underline{n}_t(J)} (x, \tilde{y})\vert $ $\vert y''-y'\vert  $ for some $\tilde{y }\in[y',y'']$. Thus,  since $n- \underline{n}_t(J)\leq N_t(I)$, we get
$$
 \frac{\vert \Delta\Phi_{\underline{n}_t(J)}(J)  - \Delta\Phi_{n }(J) \vert  }{ \Delta\Phi_{n }(J)}    \leq \frac{ N_t(I) \overline{\vert \Phi'\vert }\vert y''-y'\vert }{ \Delta\Phi_{n }(J) },
 $$ 
 which, using that $\frac{ N_t(I)}{ \Delta\Phi_{n }(J)} \leq 2 $ by~(\ref{asymptotics curves}), is less than 
 $\epsilon_0$ by~(\ref{size J}).

Let us finally prove~(\ref{distortion}). Remark that $J^{h}_n$ is an interval since $\Phi_{n }$ is monotone by Lemma~\ref{large all involved n}.  
 Since by mean value theorem there exists $\eta_1, \eta_2 \in [y',y'']$ such that $h= \vert  \Phi_{n }'(x,\eta_1) \vert 
\Leb(J^{h}_n)$ and $ \Delta \Phi_{n } (J ) =  \vert  \Phi_{n }'(x,\eta_2) \vert   \vert y''-y'\vert $,~(\ref{distortion}) follows if we prove that $\vert \frac{\vert  \Phi_{n }'(x,\eta_2) \vert }{\vert  \Phi_{n }'(x,\eta_1) \vert }-1\vert \leq \epsilon_0$. 
Let us show that this holds by showing that 
${ \max_{y'\leq y \leq y''} \vert \Phi''_n(x,y)\vert  \vert y''-y'\vert } \leq $ \ $\epsilon_0  {\min_{y'\leq y \leq y''} \vert \Phi'_n(x,y)\vert }$. This follows from~(\ref{size J})
since $\max_{y'\leq y \leq y''} \vert \Phi''_n(x,y)\vert  $ \ $ \leq \frac{4\pi d^2}{ \delta_0^2}\, \Delta\Phi_{\underline{n}_t(x)}(I) $ by~(\ref{all controlled second}) and 
$\min_{y'\leq y \leq y''} \vert \Phi'_n(x,y)\vert  \geq \frac{\delta_0}{2d} \Delta\Phi_{\underline{n}_t(x)}(I)$ by~(\ref{all large derivative}) and~(\ref{ratios stretch}).
\end{proof}

\subsection{Final equidistribution estimates.}\label{area estimates sec}
Let us use  the properties in Lemma~\ref{properties}  to show that, for each $t\geq t_0$,  
 each $J=\{x\} \times [y',y'']$ with $x \in X(t)$ and $[y',y''] \in \xi(x,t)$ verifies the equidistribution estimate~(\ref{mixing interval estimate}) in  Lemma~\ref{mixingcriterium}.  

Let us first prove that, if $Q = [x_1,x_2]\times[y_1,y_2]\times [0,h]$ is the cube fixed at the beginning of \S~\ref{parameters} and
$\Leb$ in the LHS denotes the  $1$-dimensional Lebesgue on the fiber $\{x\} \times  \T$, we have  
\be\label{areabound}
\Leb ( \{x\} \times  [y',y''] \, \cap\, f^{\Phi}_{-t} Q) \geq 
 \sum_{n= \underline{n}_t(J)}^{\overline{n}_t(J) }  \chi
\left( f^n(x,y') \right) \Leb \left( J_n ^{h} \right), 
\ee
where $J_n ^{h}$ was defined in~(\ref{J_n^h def}) and $\chi$ is the smoothened characteristic function of the base of $Q$ defined in~(\ref{N_0}), \S~\ref{parameters}.  
Let us remark that by definition of $J_n ^{h}$, if $y \in J_n ^{h}$, then   $\Phi_n(x,y)<t$ but $\Phi_{n+1}(x,y)\geq t$ since  $h< \underline{\Phi}$, so $n_t(x,y) = n$ and  $\{x \} \times J_n ^{h}$ is contained in $n_t(x,y)=n$. Thus, the intervals $J_n ^{h}$, $ \underline{n}_t(J) \leq n \leq \overline{n}_t(J) $, are all disjoint.  Hence, to prove~(\ref{areabound}), it is enough to show that if $\chi
\left( f^n(x,y') \right) >0$, then  $\{x \} \times J_{n} ^{h} \subset  \{x\} \times  [y',y'']  \, \cap \,f^{\Phi}_{-t} Q$. If $y \in J_n ^{h}$,  since as we remarked $n_t(x,y) = n$, we have  by definition of special flow~(\ref{susp flow def}) that
$ f^\Phi_t\left( (x,y),0\right) = \left(f_n(x,y) , t-   \Phi_n(x,y)   \right)$ with  $0 <  t-   \Phi_n(x,y) < h$ by definition of $J^h_n$.  
If $\chi(f^n(x,y'))>0 $, by definition of $\chi$ (see~(\ref{N_0}) in~\S~\ref{parameters}),  $f^n(x,y') \in [x_1, x_2 ]\times [y_1, y_2 -\frac{\epsilon_0}{2}(y_2-y_1)]$. Since $\vert y'-y\vert \leq \epsilon_0(y_2-y_1)/2$ 
by~(\ref{size J})  and $f^n$ 
preserves $y$-fibers and  distances between points in a $y$-fibers, we also have  $f^n(x,y) \in [x_1, x_2 ]\times [y_1, y_2]$.  This shows that $f^{\Phi}_t (\{ x \} \times J_{n} ^{h} ) \subset Q$ and concludes the proof of~(\ref{areabound}). 

Let us now estimate the RHS of~(\ref{areabound}). For $t\geq \overline{t}$, using~(\ref{asymptotics curves},~\ref{variation slopes}, ~\ref{distortion})
and then~(\ref{equidistribution base}), we get 
\begin{eqnarray}
 & & 
 \sum_{n= \underline{n}_t(J)}^{\overline{n}_t(J) }   \chi
\left( f^n(x,y') \right) \Leb \left( J_n ^{h} \right) = \nonumber \\ &  & = \frac{ \sum_{n= \underline{n}_t(J)}^{\overline{n}_t(J) }\chi
\left( f^n(x,y') \right) \frac{\Delta n_t(J)}{\Delta \Phi_{\underline{n}_t(J)} (J) }    \frac{\Delta \Phi_{\underline{n}_t(J)} (J)} {\Delta \Phi_n (J) } \frac{\Delta \Phi_n (J)
\Leb \left( J_n ^{h} \right)}{h \vert y''-y'\vert }    h\vert y''-y'\vert   } {\Delta n_t(J) }  \geq   \nonumber \\ 
& &\geq  (1-\epsilon_0)^5  h \vert y''-y'\vert (x_2-x_1) (y_2-y_1 )  \label{use equidistribution} .\nonumber
\end{eqnarray}
This, together with~(\ref{areabound}), concludes the proof of~(\ref{mixing interval estimate}) by choice of  $\epsilon_0$ in~(\ref{epsilon_0}), \S~\ref{parameters}.  
Since the partitions $\xi(x,t)$ are by construction a subdivision of the partition $\xi_1(x,t)$, we already verified in Lemma~\ref{xi1} that the partitions satisfy also the first assumption~(\ref{measure partitions}) of  Lemma~\ref{mixingcriterium}. Thus, mixing of $f^{\Phi}$ follows from Lemma~\ref{mixingcriterium}, concluding the proof of Theorem~\ref{mixing}.

\section{Cocycle Effectiveness}
\label{sec:effectivenessproof}
In this section we prove Theorem~\ref{thm:effectiveness}. We begin by recalling basic results on the cohomological equation for
the skew-shift  essentially due to A.~Katok at the beginning of the 80's
(published in \cite{Katok:CC}, \S 11.6.1). 

Let us consider the \emph{cohomological equation }for a linear skew-shift of the form 
(\ref{eq:skewshift}), that is,  the linear difference equation 
\begin{equation}
\label{eq:CE}
u\circ f - u= \Phi\,,
\end{equation}
for a given function $\Phi$ on $\T^2$. By the decomposition $\Phi := \phi + \phi^\perp$ of any
function $\Phi\in L^2(\T^2)$ into a sum of a function $\phi\in \pi^\ast L^2(\T)^\perp$ and a function
$\phi^\perp \in \pi^\ast L^2(\T)$, the equation can be decomposed into the cohomological equation 
$$
u^\perp \circ R_\alpha -  u^\perp = \phi^\perp
$$
for the rotation $R_\alpha$ of the circle of angle $\alpha\in \R$ and the cohomological
equation~(\ref{eq:CE}) with a right hand side satisfying the property
\begin{equation}
\label{eq:zerofiberaverage}
\int_ \T \Phi(x, y) dy = 0 \,, \text{ \rm for all } x\in \T\,.
\end{equation}
The space $L^2(\T^2)$ further decomposes into orthogonal irreducible components
for the action of the skew-shift.  Let $A\in GL(2, \R)$ be the matrix
$$
A = \begin{pmatrix}  1 & 1 \\ 0  & 1 \end{pmatrix} \,.
$$
Let $\{ e_{m,n} :  (m,n) \in \Z^2\}$ be the standard Fourier basis of  $L^2(\T^2)$, 
that is, 
$$ 
e_{m,n}(x,y) := \exp [2\pi i (mx +ny) ]  \, , \quad \text{ \rm for all } (x, y) \in \T^2\,.
$$
Let $\mathcal O_A$ be the set of orbits of the action of the matrix $A$ on $\Z^2$.
For any $\omega\in \mathcal O_A$, let $H_\omega \subset L^2(\T^2)$ be the subspace defined
as follows:
$$
H_\omega =   \bigoplus_{(m,n) \in \omega}  \C e_{m,n} \,.
$$
The following result is well-known and easy to verify:
\begin{lemma} The space $L^2(\T^2)$ admits an orthogonal splitting
$$
L^2(\T^2) =  \bigoplus_{\omega\in \mathcal O_A} H_\omega\,;
$$
all the components $H_\omega$ are invariant under the skew-shift $f:\T^2 \to \T^2$,
that is,
$$
f^\ast ( H_\omega ) = H_\omega   \,, \quad \text{ \rm for all }\omega \in \mathcal O_A\,.
$$
\end{lemma}
The existence of solutions of the cohomological equation can therefore be investigated in
each irreducible component $H_\omega$.  We describe below the space $\mathcal O_A$ and 
the irreducible components $H_\omega$, $\omega \in \mathcal O_A$ in more detail.

Let $(m,n)\in \Z^2$. If $n =0$, the $A$-orbit $[(m, 0)] \subset \Z^2$ of
$(m,0)$ is reduced to a single element. The space
$$
H_0 := \bigoplus_{ m\in \Z}  H_{[(m,0)]} = \pi^\ast L^2(\T)
$$
is the space of functions which factor through a square-integrable function on the circle.  For such functions the cohomological equation is reduced to the cohomological equation for circle rotations.
We are especially interested in functions in the orthogonal complement of $H_0$, that is, functions
of zero average along the fibers of the projection $\pi: \T^2 \to \T$ (see~(\ref{eq:zerofiberaverage})).

If $n \neq 0$, then the $A$-orbit $[(m, n)] \subset \Z^2$ of $(m,n)$ can is described
as follows:
$$
[(m, n)]  = \{ (m+jn, n ) :  j\in \Z \}\,.
$$ 
It follows that every $A$-orbit can be labeled uniquely by a pair $(m,n) \in \Z_{\vert n\vert} \times
 \Z\setminus\{0\}$. Let $H_{(m,n)}$ denote the corresponding factor.
Let $C^\infty(H_{(m,n)})$ be the subspace of smooth functions in $H_{(m,n)}$. By definition
every function $\Phi\in C^\infty(H_{(m,n)})$ has a Fourier expansion of the form
$$
 \Phi =  \sum_{j\in \Z} \Phi_j e_{m+jn,n} \,.
$$
For every $s>0$, let $W^s (H_{(m,n)})$ be the standard Sobolev space, that is the completion 
of $C^\infty(H_{(m,n)})$ with respect to the norm:
$$
\Vert \Phi \Vert_s :=  \left(\sum_{j\in \Z}  (1+(m+ jn)^2 +n^2) ^s \vert \Phi_j\vert^2\right)^{1/2}\,.
$$

\begin{thm} (\cite{Katok:CC}, Th. 11.25) 
\label{thm:smoothcb}
There exists a unique distributional obstruction to the existence of a smooth solution $u\in C^{\infty} (H_{(m,n)})$ of  the cohomological equation~(\ref{eq:CE}) with right hand side $\Phi \in 
C^{\infty} (H_{(m,n)})$. Such an obstruction is the invariant distribution $D_{(m,n)} \in W^{-s} (\T^2)$
for all $s>1/2$ defined as follows:
$$
D_{(m,n)}  ( e_{a, b} ) := 
\begin{cases} 
e^{-2\pi i [ (\alpha m +\beta n) j  +  \alpha n\binom{j}{2}] }  \quad &\text { \rm if }  (a, b) =(m+jn, n) \,; \\
   0                                      \quad  &\text{ \rm otherwise} .
\end{cases}
$$
The solution of the cohomological equation for any $\Phi \in C^{\infty} (H_{(m,n)})$ such that
$D_{(m,n)} (\Phi) =0$ is given by the following formula. If $\Phi = \sum_{j\in \Z}  \Phi_j e_{m+jn, n}$,
the solution $u= \sum_{j\in \Z}  u_j e_{m+jn, n}$ is:
\begin{equation}
\begin{aligned}
u_j &= - e^{2\pi i [(\alpha m +\beta n)j + \alpha n \binom{j}{2}]}  \sum_{k=-\infty}^j \Phi_k 
e^{ -2\pi i[(\alpha m +\beta n) k + \alpha n \binom{k}{2}]} \\ &= 
 e^{2\pi i [(\alpha m +\beta n)j + \alpha n \binom{j}{2}]}  
 \sum_{k=j+1}^{\infty} \Phi_k e^{ -2\pi i[(\alpha m +\beta n) k + \alpha n \binom{k}{2}]}  \,.
\end{aligned}
\end{equation}

If $\Phi \in W^s (H_{(m,n)})$ for any $s>1$ and $D_{(m,n)} (\Phi) =0$, then the above solution
$u\in W^t  (H_{(m,n)})$ for all $t<s-1$ and there exists a constant $C_{s,t} >0$ such that
$$
\Vert u\Vert _t \leq C_{s,t}  \, \Vert  \Phi \Vert_s \,.
$$
\end{thm} 

The results below establish the quantitative behavior of ergodic averages for smooth 
functions under the skew-shift.

\begin{lemma} 
\label{lemma:L2bounds}
Let  $(m,n)\in \Z_{\vert n\vert } \times \Z\setminus\{0\}$ and let $s>1/2$.
There exists a constant $C_s>0$ such that, for any  $\Phi \in W^s(H_{(m,n)})$, 
\begin{equation}
\begin{aligned}
 C_s^{-1} \vert D_{(m,n)} (\Phi) \vert  \leq   &\liminf_{N\to +\infty}    \frac{ 1} {N^{1/2}}  \Vert \sum_{k=0}^{N-1}  \Phi \circ f^k 
\Vert_{L^2(\T^2)}  \\ 
&\leq \limsup_{N\to +\infty}    \frac{ 1} {N^{1/2}}  \Vert \sum_{k=0}^{N-1}  \Phi \circ f^k 
\Vert_{L^2(\T^2)} \leq C_s \vert D_{(m,n)} (\Phi) \vert\,.
\end{aligned}
\end{equation}
\end{lemma}
\begin{proof}
Let us write the Fourier expansion of a function $\Phi \in W^s(H_{(m,n)})$ and directly
compute the ergodic sums. We obtain the formula
$$
\Vert \sum_{k=0}^{N-1}  \Phi \circ f^k 
\Vert^2_{L^2(\T^2)} =  \sum_{\ell\in \Z}  \vert \sum_{ j=\ell-N+1}^\ell \Phi_j 
e^{-2\pi i [(\alpha m +\beta n) j + \alpha n\binom{j}{2}] }\vert^2
$$
from which the result follows. Let us first prove the lower bound, which is the relevant one for our paper. 
Since $\Phi \in W^s(H_{(m,n)})$, by H\"older inequality,
\begin{equation}
\label{eq:tailbound}
\vert \sum_{\vert  j \vert \geq M} \Phi_j e^{-2\pi i [(\alpha m +\beta n) j + \alpha n\binom{j}{2}]}\vert
 \leq  K_s \Vert \Phi \Vert_s  M^{-(s-1/2)} \,, \text{ for any } M\in \N \setminus \{0\}\,. 
\end{equation}
It follows that there exists a constant $K'_s>0$ such that 
\be
\begin{aligned}
\frac{1}{N} \sum_{\ell=N/4}^{N/2}  \vert \sum_{ j=\ell-N+1}^\ell \Phi_j 
& e^{-2\pi i [(\alpha m +\beta n) j + \alpha n\binom{j}{2}] }\vert^2  \\ 
&\geq \, \frac{ \vert D_{(m,n)}(\Phi) \vert^2}{8} -  K'_s  \Vert \Phi \Vert_s^2  N^{-2(s-\frac{1}{2})} \,,
\end{aligned}
\ee
which implies the lower bound on the lower limit claimed in the statement.

\smallskip
As for the upper bound, it can be proved as follows. For any $0 <\eta<1$,  
the following bound can be derived from the estimate in formula~(\ref{eq:tailbound}) :
\be
\begin{aligned}
\frac{1}{N} \sum_{\ell=N^\eta}^{N-N^\eta}  \vert \sum_{ j=\ell-N+1}^\ell \Phi_j 
& e^{-2\pi i [(\alpha m +\beta n) j + \alpha n\binom{j}{2})] }\vert^2  \\ 
& \leq \,  2\vert D_{(m,n)}(\Phi) \vert^2 + K'_s \Vert \Phi \Vert_s^2  N^{-2\eta(s-\frac{1}{2})}\,.
\end{aligned}
\ee
By applying again formula~(\ref{eq:tailbound}) we can derive the following bounds:
\begin{equation}
\label{eq:upperbound1}
\begin{aligned}
\frac{1}{N} &\sum_{\ell\geq  N +N^\eta} \vert \sum_{ j=\ell-N+1}^\ell \Phi_j 
e^{-2\pi i [(\alpha m +\beta n) j + \alpha n\binom{j}{2})] }\vert^2 \leq 
K'_s \Vert \Phi \Vert_s^2   N^{-2\eta(s-\frac{1}{2})} \,; \\
\frac{1}{N} &\sum_{\ell\leq  -N^\eta} \vert \sum_{ j=\ell-N+1}^\ell \Phi_j 
e^{-2\pi i [(\alpha m +\beta n) j + \alpha n\binom{j}{2})] }\vert^2 \leq 
K'_s \Vert \Phi \Vert_s^2   N^{-2\eta(s-\frac{1}{2})} \,.
\end{aligned}
\end{equation}
Finally the following estimates hold:
\begin{equation}
\label{eq:upperbound2}
\begin{aligned}
\frac{1}{N} &\sum_{\ell= N -N^\eta}^{N +N^\eta} \vert \sum_{ j=\ell-N+1}^\ell \Phi_j 
e^{-2\pi i [(\alpha m +\beta n) j + \alpha n\binom{j}{2})] }\vert^2 \leq 
2K'_s \Vert \Phi \Vert_s^2   N^{-(1-\eta)} \,; \\
\frac{1}{N}& \sum_{\ell= -N^\eta}^{N^\eta} \vert \sum_{ j=\ell-N+1}^\ell \Phi_j 
e^{-2\pi i [(\alpha m +\beta n) j + \alpha n\binom{j}{2})] }\vert^2 \leq 
2K'_s \Vert \Phi \Vert_s^2   N^{-(1-\eta)} \,.
\end{aligned}
\end{equation}
The upper bound on the upper limit claimed in the statement follows immediately
from the estimates~(\ref{eq:upperbound1}) and~(\ref{eq:upperbound2}).
\end{proof}

The uniform norm of the ergodic averages of sufficiently smooth functions can be controlled
sharply along a subsequence of times. This result can be derived from classical (sharp) number
theory results on Weyl sums of quadratic polynomials (see \cite{FJK:Weyl} and references therein
or \cite{Mr}) or from the results of \cite{FlaFor:nil} on the quantitative equidistribution of Heisenberg nilflows.

\begin{thm}  
\label{thm:unifbound}
Let $\alpha \in \R\setminus \Q$ be any irrational number and let $ s > 3$. There exist a constant 
$M_s >0$ and a (positively) diverging sequence $\{N_\ell\}_{\ell\in \N}$ (depending on 
$\alpha$) such that, for all  $\Phi  \in W^s(\T^2) \cap H_0^\perp$ and for all $(x,y)\in \T^2$,
\be
\label{eq:unifbound}
\frac{ 1} {N_\ell^{1/2}}\,\vert \sum_{k=0}^{N_\ell-1}  \Phi \circ f^k(x,y) \vert  \leq M_s \Vert \Phi \Vert_s\,.
\ee
\end{thm}
\begin{proof} Since the special flow with a constant roof function $r>0$ of a uniquely ergodic linear 
skew-shift is smoothly equivalent to a uniquely ergodic Heisenberg nilflow, it sufficient to prove
the result for Heisenberg nilflow. In fact, let $\{f^r_t\}$ denote the special flow over $f$ with roof function
$r>0$. Let $\chi \in C_0^\infty (0, r)$ be compactly supported function of integral equal to
$1$ on $\R$. For any function $\Phi\in \T^2$, let $\hat \Phi_\chi :  \T^2\times [0,r] \to \R$ be the
smooth function defined  as follows:
$$
\hat \Phi_\chi (x,y, z) =   \Phi(x,y) \chi(z) \,, \quad \text{ for all } \, (x,y)\in \T^2, \, z\in [0,r] \,. 
$$
Since $\T^2\times [0,r]$ is a fundamental domain for the quotient $\T^2\times \R / \sim_r$
the function $\hat \Phi_\chi$, which vanishes at the boundary with all its derivatives,  projects
to a well-defined function $\Phi_\chi$ on $M\approx\T^2\times \R / \sim_r$.  The function $\Phi_\chi \in  W^s(M) \cap \pi^\ast L^2(\T^2)^\perp $ if and only if  $\Phi \in W^s(\T^2) \cap H_0^\perp$. 
By construction, since the function $\chi \in C^\infty_0(0, r)$ has integral equal to $1$ on $(0,r)$, 
 for all $N\in  \N$ and for all $(x,y)\in \T^2$, we have
\be
\label{eq:discont}
 \sum_{k=0}^{N-1}  \Phi \circ f^k (x,y)=\int_0^N   \Phi_{\chi} \circ f^r_t (x,y,0) dt \,.
\ee
Thus the statement of the theorem can be derived from the following claim. For every $s>3$ and
for every uniquely ergodic Heisenberg nilflow $\{\phi_t^W\}$ on $M= \Gamma \backslash \Heis$, 
there exist a constant $C_s >0$ and a (positively) diverging sequence $\{T_\ell\}\subset \R$ such that,
for all  $\Psi  \in W^s(M) \cap \pi^\ast L^2(\T^2)^\perp$ and for all $x\in M$,
\begin{equation}
\label{eq:flowbound}
\frac{ 1} {T_\ell^{1/2}}\, \vert\int_0^{T_\ell}  \Psi \circ \phi_t^W (x) dt \vert \leq 
C_s \Vert \Psi \Vert_s \,.
\end{equation}
The above claim follows from Lemma 5.5 and Lemma 5.8 in \cite{FlaFor:nil}. Let
$\bar W=  (1, \alpha)\in \R^2$ be the projection of the generator $W \in \heis$
onto the abelianized Lie algebra $\heis/ [\heis, \heis]\approx \R^2$. For any compact set
$K \subset PSL(2,\Z) \backslash PSL(2,\R)$ there exists a constant $C_s:=C_s(K)$
such that the bound~(\ref{eq:flowbound}) holds under the condition that $\log T_\ell 
\in \R^+$ is a return time to $K$ of the trajectory of the point
$$
PSL(2,\Z) \begin{pmatrix} 1 & \alpha \\ 0 & 0
\end{pmatrix}   \in PSL(2,\Z) \backslash PSL(2,\R)
$$
under the geodesic flow on the unit tangent bundle  $PSL(2,\Z) \backslash PSL(2,\R)$
of the modular surface. Thus the above claim follows from the recurrence of all irrational 
points of  $PSL(2,\Z) \backslash PSL(2,\R)$ under the modular geodesic flow.

Let $\{T_\ell\}\subset \R^+$ be any diverging sequence such that the bound~(\ref{eq:flowbound})
holds. By the identity~(\ref{eq:discont}), the bound~(\ref{eq:unifbound}) holds for the diverging
sequence $\{ [T_\ell]\} \subset \N$. The proof of the theorem is completed.
\end{proof}

\begin{rem}\label{highereffectiveness_rem2} The theory on the existence of smooth solutions of the cohomological equation 
outlined above generalizes to skew-shifts in any dimensions (in fact, to nilflows on any nilpotent
manifold \cite{FlaFor:CEnil}). However, as far as we know, Theorem~\ref{thm:unifbound} is not 
established for higher dimensional skew-shifts, not even for typical rotation numbers. Bounds
on ergodic averages of higher dimensional skew-shifts are closely related to bounds on
Weyl sums for polynomials of degree greater than $2$.
\end{rem}

We conclude by proving Theorem~\ref{thm:effectiveness} which states that  any sufficiently smooth function $\Phi\in H_0^\perp$ is a smooth coboundary for a uniquely ergodic (irrational) skew-shift  
if and only if it is a measurable coboundary.

\begin{proof} [Proof of Theorem~\ref{thm:effectiveness}]  Let  $\{\Phi_\ell \}$ denote the sequence of the ergodic sums of the function $\Phi \in  W^s(\T^2) \cap H_0^\perp$, that is, 
$$
\Phi_\ell (x,y) =  \sum_{k=0}^{N_\ell-1}  \Phi \circ f^k(x,y)  \,, \quad 
\text{ for all } (x,y)\in \T^2\,,
$$
along the sequence $\{N_l\}_{l\in \N}$ constructed in Theorem~\ref{thm:unifbound}.  
Let $S^\ell_\epsilon \subset \T^2$ be the set defined as follows:
\begin{equation}
\label{eq:Seps}
S^\ell_\epsilon := \{ (x,y) \in \T^2 :  \vert \Phi_\ell (x,y)  \vert \geq \epsilon N_\ell ^{1/2}\}\,.
\end{equation}
Theorem~\ref{thm:unifbound} implies by an elementary estimate that
$$
\Vert \Phi_\ell \Vert^2_{L^2(\T^2)} \leq  M_s^2 \Vert \Phi\Vert_s^2 \Leb (S^\ell_\epsilon) N_l
+ \epsilon^2 N_\ell  (1- \Leb (S^\ell _\epsilon))\,.
$$
It follows that, if the function $\Phi$ does not belong to the kernel of all invariant distributions, 
by Lemma~\ref{lemma:L2bounds} there exists a constant $c_\Phi>0$ such that
$$
c_\Phi N_\ell \leq M_s^2 \Vert \Phi\Vert_s^2 \Leb (S^\ell _{\epsilon}) N_\ell
+ \epsilon^2 (1- \Leb (S^\ell _{\epsilon} ))  N_\ell  \,,
$$
hence
$$
(c_\Phi-\epsilon^2)  \leq (M_s^2 \Vert \Phi\Vert_s^2 -\epsilon^2) \Leb (S^\ell_\epsilon) 
$$
and there exist $\epsilon >0$ and $\eta(\epsilon)>0$ such that
\begin{equation}
\label{eq:measlb}
\Leb (S^\ell _\epsilon) \geq  \eta_{\epsilon} \,, \quad \text {\rm for all } \ell \in \N\,.
\end{equation}
We conclude the argument by proving that if the lower bound~(\ref{eq:measlb}) holds,
the function $\Phi$ is not a measurable coboundary. In fact, let us assume it is 
and derive a contradiction.  Let  $u$ be a measurable transfer function  
on $\T^2$. Since $u$ is almost everywhere finite, there exists a constant $M_\epsilon>0$ such that 
$$
\Leb \{ (x,y) \in \T^2 :  \vert u(x,y ) \vert  \leq M_\epsilon/2 \} \geq 1- \eta(\epsilon)/4\,.
$$
Thus, by the identity  $\Phi_\ell (x,y) = u\circ f^{N_\ell} (x,y) -  u(x,y)$,  it follows that 
\begin{equation}
 \Leb\{ (x,y )\in \T^2 : \vert \Phi_\ell (x,y) \vert  \leq M_\epsilon \} 
  \geq 1-\eta(\epsilon)/2\,;
 \end{equation}
however,  by definition~(\ref{eq:Seps})  the subsets $S_\ell (\epsilon)$ and  $\{ (x,y )\in \T^2 :  
\vert \Phi_\ell (x,y) \vert  \leq M_\epsilon \}$ are disjoint for all $N_\ell > \epsilon^{-2} M^2_\epsilon$, hence 
$$
1 + \eta(\epsilon)/2 \leq  \Leb (S_\ell (\epsilon))  +  
\Leb\{ (x,y )\in \T^2 : \vert \Phi_\ell (x,y) \vert  \leq M_\epsilon \}  \leq 1\,,
$$
which is the desired contradiction. 
\end{proof}

\smallskip

Let us show that the class $\MM$ of mixing roof functions in Definition~\ref{mixing_class} contains the complement of a countable codimension subspace of a dense subspace of the space of smooth  functions which can be described explicitely. 

\begin{cor}\label{complementkernels}
The class $\MM$ contains the set 
$$\mathcal{P}_{\T^2} \backslash \left(  \cap_{n \in \Z \backslash \{0\} } \cap_{m \in \Z_{\vert n\vert }}  \text{\rm ker}\, D_{(m,n)} \right) \cap  \text{\rm ker}\, D_{(0,0)}\,,
$$ 
where $\mathcal{P}_{\T^2}$ denotes the space of all real-valued functions on $\T^2$ which are   trigonometric polynomials in both variables and $D_{(m,n)}$ are the invariant distributions
for the linear skew-shift described  in Theorem~\ref{thm:smoothcb}. 
\end{cor}
\begin{proof}
Since $\Phi \in \mathcal{P}_{\T^2}^+ $ is a trigonometric polynomials in all variables,  
the inclusion $\mathcal{P}_{\T^2}^+ \subset \RR$ holds. By Theorem~\ref{thm:effectiveness} and Theorem~\ref{thm:smoothcb}, $\phi$ is a measurable coboundary if and only if it is not in the kernel of all invariant distributions in Theorem~\ref{thm:smoothcb}.
\end{proof}

Let us now prove that the roofs functions of the examples  at the end of \S~\ref{sec:mixspecialflows} belong to the class $\RR$. We will prove that roofs in (3) are in $\RR$, since (1), (2) have analogous proofs. Let  $\Phi$ be as in (3).  By Corollary~\ref{complementkernels}, it is enough to find a distribution $D_{(m,n)}$ as in Theorem~\ref{thm:smoothcb} which is not zero. One can check that the roof function 
$\Phi \in H_{[(0,1)]} + H_{[(0,-1)]}$ and, by Theorem~\ref{thm:smoothcb} and by assumption,
$$
D_{(0,1)}( \sum_{j \in \Z}a_j e^{2\pi i (jx+y)} ) = \sum_{j \in \Z} a_j e^{-2\pi i (\beta j+ \alpha \binom{j}{2} )}   \neq 0 \,.
$$

\section{Non-triviality, weak mixing and mixing equivalences}\label{proof:mainthms}
In this section we give the proofs of the equivalences in Theorem~\ref{thm:mainflows} and Theorem~\ref{thm:main}.
We first recall for the convenience of the reader the following well-knwon elementary result about special flows
 that relates non-triviality of time-changes and weak mixing (see for instance \cite{Katok:CC}, \S 9.3.4).

\begin{lemma}[Non-triviality and weak mixing]\label{lemma:wm} Let $f$ be a measure preserving transformation on a probability space
$(\Sigma, \nu)$. For any measurable almost coboundary $\Phi:\Sigma \to \R^+$, the special flow $f^\Phi$ 
over $f$ with roof function $\Phi$ is measurably trivial, 
hence it is not weak mixing.
\end{lemma}

\begin{proof} Since $\Phi$ is an almost coboundary,   there exist a constant $C_\Phi >0$ and a measurable 
function $u:X \to \R$ such that
\begin{equation}
\label{eq:meascob}
\Phi - C_\Phi =  u\circ f - u\,,
\end{equation}
Let $I: X \times \R \to X\times \R$ be the map
$$
I(x,z )= (x, z+u(x) )  \,, \quad \text{ \rm for all } (x,z) \in X \times \R\,.
$$
It is immediate to see that the map $I$ is a measurable isomorphism of $X\times \R$
which conjugates the vertical flow to itself.  Since  the phase space of the special flow under $\Phi$ is defined as the 
quotient $X/\sim_\Phi$ with respect to the equivalence relation $(x,\Phi(x)+z) \sim_\Phi (f(x), z)$, for all $x\in X, z\in \R$, it is sufficient to prove that the map $I$ has a well-defined 
projection on the quotient spaces $X/\sim_\Phi$ and $X/\sim_{C_\Phi}$.
Since $u$ is a solution of the cohomological equation~(\ref{eq:meascob}), the following 
identities hold:
\bes
\begin{aligned}
I (x, \Phi(x)+z) &= (x, \Phi(x) +u(x)+z) =\\ & = (x,  C_\Phi + u\circ f (x)+z ) \sim_{C_\Phi} (f(x),  u\circ f (x)+z ) = I(f(x), z) \,,
\end{aligned}
\ees
hence the map $I: X \times \R \to X\times \R$ passes to the quotient as claimed.
It is well known and immediate to verify that no special flow with constant roof
function is weak mixing.
\end{proof}

\begin{proofof}{Theorem}{thm:main}
Let $\MM= \RR \setminus \TT$ be the class defined in Definition~\ref{mixing_class}.  As a consequence of the cocycle effectiveness (Theorem~\ref{thm:effectiveness}), $\TT $ is in fact the intersection of the dense space $\RR$ with the kernel of countable many linear functionals, as stated in Theorem~\ref{thm:main} (see Corollary~\ref{complementkernels}). Let us prove the equivalences of $(1)-(4)$.  The implication $(1)\Rightarrow (4)$ is exactly the content of Theorem~\ref{mixing}. The implication $(4)\Rightarrow (3)$ is obvious.  If $\Phi$ is smoothly trivial, hence in particular measurably trivial, $f^\Phi$ is \emph{not} weak mixing (see Lemma~\ref{lemma:wm}). Thus, taking counterpositives,  $(3)\Rightarrow (2)$. We are left to prove $(2)\Rightarrow (1)$. Let us again prove the counterpositive implication and, since  $\Phi \in \RR$ by the assumptions in Theorem~\ref{thm:main}, if $(1)$ does not hold, we know that $\Phi \in \RR\backslash \MM $. This means, by Definition~\ref{mixing_class} of $\MM$, that the projection $\phi$ defined in~(\ref{phidef}) is a measurable coboundary for $f$.    Since clearly $\RR\subset W^s(\T^2)$, $s>3$, by Theorem~\ref{thm:effectiveness} we then know that $\phi$ belongs to the kernel of all $f$-invariant distributions and it is a \emph{smooth} almost coboundary. 
It is easy to check solving the cohomological equation in Fourier coefficients that any trigonometric polynomial on $\T$ is a smooth almost coboundary for any irrational circle rotation. Thus, $\Phi = \phi + \phi^\perp$ is a smooth almost coboundary for the skew-shift and $f^\Phi $ is smoothly trivial, or equivalently, $(2)$ does not hold. This concludes the proof of the equivalences.
\end{proofof}

\begin{proofof}{Theorem}{thm:mainflows} We will deduce Theorem~\ref{thm:mainflows} from Theorem~\ref{thm:main}. 
As summarized in~\S~\ref{sec:returnmaps}, any uniquely ergodic Heisenberg nilflow  $\phi^W$
has a global transverse smooth transverse surface $\Sigma \approx \T^2_E= \R^2/ (\Z \times \Z/E)$ and the Poincar\'e map $P_W:\T^2_E\to \T^2_E$ is a uniquely ergodic skew-shift (over a circle rotation).  Let  
$\T^2= \R^2 / \Z^2$. It follows from Lemma~\ref{lemma:returnmap} that there exist $\alpha \in 
\R \setminus \Q$ and $\beta \in \R$ such that the \emph{linear skew-shift} over a circle rotation,
defined in~(\ref{eq:skewshift}),
 is a covering map of finite order $E \in \N\setminus\{0\}$ of the Poincar\'e map, in the sense that
 the canonical projection $\pi_E : \T^2 \to \T^2_E$ yields a semi-conjugacy  between the skew-shift 
 $f$ on $\T^2$ and the Poincar\'e map $P_W$ on $\T^2_E$. It is sufficient to prove the theorem  in the 
 particular case $E=1$, when the Poincar\'e map is isomorphic to a uniquely ergodic standard 
 skew-shift of the form~(\ref{eq:skewshift}).  In fact, all other cases can be treated similarly or reduced to this one by considering the appropriate covering map on $\T^2$.
 
Let us say that a positive function $\alpha$ belongs to  the class $\mathcal{A}$ (respectively to the class $\MM$) iff  the return time function $\Phi^{\alpha}$ given by Lemma~\ref{lemma:returntimefs} where $\Phi\equiv 1$ belongs to $\RR$ (respectively to $\MM$). The proof of Theorem~\ref{thm:mainflows} now reduces simply in a rephrasing $(1)-(4)$ in Theorem~\ref{thm:mainflows} using the dictionary between time-changes of flows and special flows recalled in~\S~\ref{sec:mixspecialflows} and checking that they correspond to $(1)-(4)$ in Theorem~\ref{thm:main}.
\end{proofof}

\section*{Acknowledgements}
We would like to thank University of Maryland for the hospitality during the visit when part of this work was completed. The second author acknowledges support of the NSF grant DM 0800673. The last author is currently supported by an RCUK Academic Fellowship, whose support is fully
acknowledged.

\bibliographystyle{plain}

\begin{thebibliography}{1}

\bibitem{AGH}
Louis~Auslander, Leon~W.~Green and  Frank~Hahn.
   \newblock {\em Flows on homogeneous spaces}.
\newblock Princeton University Press, Princeton, N.J.,1963.

\bibitem{AF:wea}
Artur Avila and Giovanni Forni.
\newblock Weak mixing for interval exchange transformations and translation
  flows.
\newblock {\em Annals of Mathematics}, 165 (2): 637--664, 2007.

\bibitem{Bu} 
Marc~Burger.
\newblock Horocycle flow on geometrically finite surfaces. 
\newblock {\em Duke Math. J.} 61: 779--803, 1990.

\bibitem{CFS:erg}
Isaac~P. Cornfeld, Sergei~V. Fomin, and Yakov~G. Sinai.
\newblock {\em Ergodic Theory}.
\newblock Springer-Verlag, 1980.

\bibitem{CG:rep}
Lawrence J. Corwin and Frederick P. Greenleaf.
\newblock Representations of nilpotent Lie groups and their applications. Part I: Basic theory and examples.
\newblock {\em Cambdridge studies in advanced mathematics } 18, Cambdridge University Press, Providence, Cambridge 1990.

\bibitem{Fayad:pol} 
Bassam~R.~Fayad. 
\newblock Polynomial decay of correlations for a class of smooth flows on the two torus.
\newblock {\em Bull. Soc. Math. France} 129:  487Ð503, 2001.

\bibitem{Fayad:wm}
\bysame.
\newblock Weak mixing for reparameterized linear flows on the torus.
\newblock {\em Ergodic Theory Dynam. Systems} 22 (1): 187--201, 2002.

\bibitem{Fa:ana}
\bysame.
\newblock Analytic mixing reparametrizations of irrational flows.
\newblock {\em Ergodic Theory and Dynamical Systems}, 22 (2): 437--468, 2002.


\bibitem{FJK:Weyl} 
Heinz~Fiedler, Wolfgang~B.~Jurkat and Otto~K\"orner.
\newblock Asymptotic expansions of finite theta series.
\newblock {\em Acta Arithmetica}, XXXII, 129-146, 1977.

\bibitem{FlaFor:hor} 
Livio~Flaminio and Giovanni~Forni.
\newblock Invariant distributions and time averages for horocycle flows.  
\newblock {\em Duke Math. J.}  119 (3): 465--526, 2003.

\bibitem{FlaFor:nil}
\bysame.
\newblock  Equidistribution of nilflows and applications to theta sums.
\newblock {\em Ergodic Theory and Dynamical Systems}, 26:2:409-433, 2006. 

\bibitem{FlaFor:CEnil}
\bysame.
\newblock  On the cohomological equation for nilflows.
\newblock {\em Journal of Modern Dynamics} 1 (1): 37-60, 2007. 

\bibitem{Fo:sol}
 Giovanni Forni.
\newblock Solutions of the cohomological equation for area-preserving flows on compact surfaces of higher genus.
  flows.
\newblock {\em Annals of Mathematics}, 146 (2): 295--344, 1997.

\bibitem{Fo:dev}
\bysame.
\newblock Deviations of ergodic averages for area-preserving flows on surfaces
  of higher genus.
\newblock {\em Annals of Mathematics (2)}, 155 (1): 1--103, 2002.


\bibitem{Fu:uni}
Harry~Furstenberg.
\newblock The unique ergodicity of the horocycle flow. Recent advances in topological dynamics. 
\newblock {\em Lecture Notes in Math.}, 318: 95--115, 1973, Springer, Berlin.

\bibitem{Green} 
Leon W. Green.
\newblock Spectra of nilflows.
\newblock {\em Bull. Amer. Math. Soc. } 67: 414--415, 1961.

\bibitem{Hj} 
Dennis~A.~Hejhal.
\newblock On the uniform equidistribution of long closed horocycles. 
Loo-Keng Hua: a great mathematician of the twentieth century. 
\newblock {\em Asian J. Math.} 4 (4): 839--853, 2000. 

\bibitem{He} 
Michael~R. Herman. 
\newblock Examples de flots hamiltoniens dont aucune perturbation en topologie
$C^\infty$ n'a d'orbites p\'eriodiques sur un ouvert de surfaces d'\'energies.
\newblock {\em C. R. Acad. Sci. Paris} 312: 989--994, 1991.
      
\bibitem{Ka:int}
Anatole~B. Katok.
\newblock Interval exchange transformations and some special flows are not
  mixing.
\newblock {\em Israel Journal of Mathematics}, 35 (4): 301--310, 1980.
  
\bibitem{Katok:CC}
\bysame.
\newblock {C}ombinatorial {C}onstructions in {E}rgodic {T}heory and {D}ynamics.
\newblock University Lecture Series Vol. 30.
\newblock American Mathematical Society, Providence, RI, 2003.

\bibitem{SK:mix}
Konstantin~M. Khanin and Yakov~G. Sinai.
\newblock Mixing for some classes of special flows over rotations of the
  circle.
\newblock {\em Funktsional'nyi Analiz i Ego Prilozheniya}, 26(3):1--21, 1992.
\newblock (Translated in: \emph{Functional Analysis and its Applications},
  26:3:155--169, 1992).
  
\bibitem{Ko:abs}
Andrey~V. Ko{\v c}ergin.
\newblock The absence of mixing in special flows over a rotation of the circle and in flows on a two-dimensional torus.
\newblock {\em Dokl. Akad. Nauk SSSR  }, 205: 512-518, 1972. (Translated in \newblock {\em Soviet Math. Dokl.}, 13: 949-952, 1972)

\bibitem{Ko:mix}
\bysame.
\newblock Mixing in special flows over a shifting of segments and in smooth
  flows on surfaces.
\newblock {\em Mat. Sb.}, 96(138): 471--502, 1975.
  
  \bibitem{Kolmogorov} 
Andre\"i~N.~Kolmogorov.
\newblock  On dynamical systems with integral invariance on the torus.
\newblock {\em Dokl. Akad. Nauk SSSR (N. S.)} 93: 763-766, 1953 (in Russian).

\bibitem{Kontsevich} 
Maxim~Kontsevich.
\newblock  Lyapunov exponents and {H}odge theory, 
in ``The mathematical beauty of physics'', Saclay, 1996. 
\newblock {\em Adv. Ser. Math. Phys.} 24, World Scientific, River
Edge, NJ, 318--332, 1997.


\bibitem{Kuschnirenko:mixing}
\newblock Anatoly G. Kuschnirenko.
\newblock Spectral properties of some dynamical systems with polynomial divergence of orbits.
\newblock {\em Moscow Univ. Math. Bull} 29: 82--87, 1974.


\bibitem{Ma:hor} 
Brian Marcus.
\newblock The horocycle flow is mixing of all degrees.
\newblock {\em Invent. Math. }, 46 (3): 201--209, 1978.

\bibitem{Mr} 
Jens Marklof. 
\newblock  Limit theorems for theta sums. 
\newblock  Duke Math. J. 97 (1): 127--153, 1999.

 \bibitem{MMY:coh}
Stefano~Marmi, Pierre~Moussa and Jean-Christophe~Yoccoz.
\newblock The Cohomological Equation for {R}oth-Type Interval Exchange Maps.
\newblock {\em Journal of the American Mathematical Society}, 18 (4): 823--872, 2005.

\bibitem{Ma:int}
Howard Masur.
\newblock Interval exchange transformations and measured foliations.
\newblock {\em Annals of Mathematics}, 115: 169--200, 1982.

\bibitem{Ma:rat}
\bysame.
\newblock {\em Ergodic theory of translation surfaces. }, pages 527--547.
\newblock Handbook of Dynamical Systems, Vol.~1B. Elsevier B. V., Amsterdam., 2006.

\bibitem{Pa:hor}
 Ostap S. Parasyuk.
\newblock Flows of horocycles on surfaces of constant negative curvature.
\newblock {Uspehi Matem. Nauk (N.S.)}, 8 (3): 125--126, 1953.

\bibitem{Ra:rat}
Marina Ratner.
\newblock The rate of mixing for geodesic and horocycle flows.
\newblock {\em Ergodic Theory Dynam. Systems }, 7 (2): 267--288, 1987.
       
\bibitem{Sa} 
Peter~Sarnak.
\newblock Asymptotic behavior of periodic orbits of the horocycle flow and Eisenstein series.
\newblock {\em Comm. Pure Appl. Math.} 34: 719--739, 1981.

\bibitem{Sch:abs}
Dmitri Scheglov.
\newblock Absence of mixing for smooth flows on genus two surfaces.
\newblock {\em Journal of Modern Dynamics.}, 3 (1): 13--34, 2009.
 
\bibitem{Sta:dyn}
Alexander N. Starkov.
\newblock Dynamical Systems on Homogeneous Spaces.
\newblock {\em Translations of the American Mathematical Society}, 190, Providence, Rhode Island 2002.


 \bibitem{St} 
 Andreas~Str\"ombergsson. 
 \newblock On the uniform equidistribution of long closed horocycles. 
 \newblock {\em Duke Math. J.} 123 (3): 507--547, 2004.

\bibitem{Ul:abs} 
Corinna Ulcigrai.
\newblock  Absence of mixing in area-preserving flows on surfaces.
\newblock { Preprint on {\em arXiv:0901.4764}. To appear on {\em Annals of Mathematics}. }

\bibitem{Ul:mix}
\bysame.
\newblock Mixing for suspension flows over interval exchange tranformations.
\newblock {\em Ergodic Theory and Dynamical Systems.}, 27 (3): 991--1035, 2007.


\bibitem{Ul:wea}
\bysame.
\newblock Weak mixing for logarithmic flows over interval exchange
  tranformations.
\newblock {\em Journal of Modern Dynamics}, 3 (1): 35--49, 2009.

\bibitem{Ve:gau}
William~A. Veech.
\newblock Gauss measures for transformations on the space of interval exchange
  maps.
\newblock {\em Annals of Mathematics}, 115: 201--242, 1982.

\bibitem{Za} 
Don~Zagier. 
\newblock Eisenstein series and the Riemann zeta function, in Automorphic Forms, Representation Theory and Arithmetic (Bombay, 1979).
\newblock Tata Inst. Fund. Res. Studies in Math. 
10, Tata Inst. Fund. Res., Bombay, 275--301, 1981.


\bibitem{Zorich} 
Anton Zorich. 
 \newblock Deviation for interval exchange transformations. 
\newblock {\em Ergodic Theory Dynam. Systems} 17: 1477--1499, 1997.



\end{thebibliography}

\end{document}